\documentclass[12pt]{amsart}
\usepackage[a4paper,left=2.5cm, right=2.5cm, top=2.5cm, bottom=2.5cm]{geometry}

\usepackage{amssymb,amsmath,mathrsfs,amsthm,amscd}

\newcommand{\R}{\mathbb{R}}
\newcommand{\Rn}{{\R^n}}
\newcommand{\Rm}{{\R^m}}
\newcommand{\C}{\mathbb{C}}
\newcommand{\N}{\mathbb{N}}
\newcommand{\T}{\mathbb{T}}
\newcommand{\eps}{\varepsilon}

\newcommand{\cS}{{\mathscr{S}}}
\newcommand{\cQ}{{\mathcal{Q}}}
\newcommand{\cB}{{\mathcal{B}}}
\newcommand{\cR}{{\mathcal{R}}}
\newcommand{\cI}{{\mathcal{I}}}
\newcommand{\cO}{{\mathcal{O}}}
\newcommand{\cC}{{\mathcal{C}}}

\newcommand{\Linfty}{{L^\infty}}
\newcommand{\Lp}{{L^p}}
\newcommand{\Lploc}{{L^p_{\text{loc}}}}
\newcommand{\Lone}{{L^1}}
\newcommand{\Loneloc}{{L^1_{\text{loc}}}}
\newcommand{\JNp}{{\mathrm{JN}_p}}

\newcommand{\xihat}{{\widehat{x_i}}}
\newcommand{\Omegatil}{{\widetilde{\Omega}}}
\newcommand{\ra}{\rightarrow}

\DeclareMathOperator*{\esssup}{ess\,sup}

\def\Tr#1{\mathrm{Tr}{#1}}
\def\BMO#1#2{\mathrm{BMO}_{#1}^{#2}}
\def\CMO#1#2{\mathrm{CMO}_{#1}^{#2}}
\def\CBMO#1#2{\mathrm{CBMO}_{#1}^{#2}}
\def\VMO#1#2{\mathrm{VMO}_{#1}^{#2}}
\def\bBMO#1{\mathrm{\bf{BMO}}{}{#1}}
\def\norm#1#2{\lVert{#1}\rVert_{#2}}

\def\Xint#1{\mathchoice
{\XXint\displaystyle\textstyle{#1}}%
{\XXint\textstyle\scriptstyle{#1}}%
{\XXint\scriptstyle\scriptscriptstyle{#1}}%
{\XXint\scriptscriptstyle\scriptscriptstyle{#1}}%
\!\int}
\def\XXint#1#2#3{{\setbox0=\hbox{$#1{#2#3}{\int}$ }
\vcenter{\hbox{$#2#3$ }}\kern-.57\wd0}}

\def\dashint{\Xint-}

\newtheorem{theorem}{Theorem}[section]
\newtheorem{lemma}[theorem]{Lemma}
\newtheorem{proposition}[theorem]{Proposition}
\newtheorem{corollary}[theorem]{Corollary}

\theoremstyle{definition}
\newtheorem{definition}[theorem]{Definition}
\newtheorem{example}[theorem]{Example}

\newtheorem{problem}[theorem]{Problem}

\theoremstyle{remark}

\numberwithin{equation}{section}

\begin{document}

\title[$\BMO{}{}$ on shapes and sharp constants]{$\bBMO{}{}$ on shapes and sharp constants}

\author[Dafni]{Galia Dafni}
\address{(G.D.) Concordia University, Department of Mathematics and Statistics, Montr\'{e}al, Qu\'{e}bec, H3G-1M8, Canada}
\curraddr{}
\email{galia.dafni@concordia.ca}
\thanks{The authors were partially supported by the Natural Sciences and Engineering Research Council
(NSERC) of Canada, the Centre de recherches math\'{e}matiques (CRM), and the Fonds de recherche
du Qu\'{e}bec -- Nature et technologies (FRQNT)}

\author[Gibara]{Ryan Gibara}
\address{(R.G.) Concordia University, Department of Mathematics and Statistics, Montr\'{e}al, Qu\'{e}bec, H3G-1M8, Canada}
\curraddr{}
\email{ryan.gibara@concordia.ca}


\date{\today}

\begin{abstract}
We consider a very general definition of $\BMO{}{}$ on a domain in $\Rn$, where the mean oscillation is taken with respect to a basis of shapes, i.e.\ a collection of open sets covering the domain. We examine the basic properties and various inequalities that can be proved for such functions, with special emphasis on sharp constants. For the standard bases of shapes consisting of balls or cubes (classic $\BMO{}{}$), or rectangles (strong $\BMO{}{}$), we review known results, such as the boundedness of rearrangements and its consequences. Finally, we prove a product decomposition for $\BMO{}{}$ when the shapes exhibit some product structure, as in the case of strong $\BMO{}{}$.
\end{abstract}

\maketitle

\section{{\bf Introduction}}

First defined by John and Nirenberg in \cite{jn}, the space $\BMO{}{}$ of functions of bounded mean oscillation has served as the replacement for $\Linfty$ in situations where considering bounded functions is too restrictive. $\BMO{}{}$ has proven to be important in areas such as harmonic analysis, partly due to the duality with the Hardy space established by Fefferman in \cite{fe}, and partial differential equations, where its connection to elasticity motivated John to first consider the mean oscillation of functions in \cite{fj}. Additionally, one may regard $\BMO{}{}$ as a function space that is interesting to study in its own right. As such, there exist many complete references to the classical theory and its connection to various areas; for instance, see \cite{ga,gra,jo,st}. 

The mean oscillation of a function $f\in \Loneloc(\Rn)$ was initially defined over a cube $Q$ with sides parallel to the axes as 
\begin{equation}
\label{osc}
\dashint_{Q}|f-f_{Q}|,
\end{equation}
where $f_Q=\dashint_{Q}f$ and $\dashint_{Q}=\frac{1}{|Q|}\int_{Q}$. A function $f$ was then said to be in $\BMO{}{}$ if the quantity (\ref{osc}) is bounded independently of $Q$. Equivalently, as will be shown, the same space can be obtained by considering the mean oscillation with respect to balls; that is, replacing the cube $Q$ by a ball $B$ in (\ref{osc}). Using either characterization, $\BMO{}{}$ has since been defined in more general settings such as on domains, manifolds, and metric measure spaces (\cite{jon, bn1, cw2}). 

There has also been some attention given to the space defined by a mean oscillation condition over rectangles with sides parallel to the axes, either in $\Rn$ or on a domain in $\Rn$, appearing in the literature under various names. For instance, in \cite{ko2}, the space is called ``anisotropic $\BMO{}{}$" to highlight the contrast with cubes, while in papers such as \cite{cs,dlowy,dlwy,fs}, it goes by ``little $\BMO{}{}$" and is denoted by $\text{bmo}$. The notation $\text{bmo}$, however, had already been used for the ``local $\BMO{}{}$" space of Goldberg (\cite{gol}), a space that has been established as an independent topic of study (see, for instance, \cite{bu,dy,ya}). Yet another name for the space defined by mean oscillations on rectangles - the one we prefer - is the name ``strong $\BMO{}{}$". This name has been used in at least one paper (\cite{le1}), and it is analogous to the terminology of strong differentiation of the integral and the strong maximal function (\cite{cf,guz,jt,jmz,sa}), as well as strong Muckenhoupt weights (\cite{br,lpr}). 

In this paper we consider $\BMO{}{}$ on domains of $\Rn$ with respect to a geometry (what will be called a basis of shapes) more general than cubes, balls, or rectangles. The purpose of this is to provide a framework for examining the strongest results that can be obtained about functions in $\BMO{}{}$ by assuming only the weakest assumptions. To illustrate this, we provide the proofs of many basic properties of $\BMO{}{}$ functions that are known in the literature for the specialised bases of cubes, balls, or rectangles but that hold with more general bases of shapes. In some cases, the known proofs are elementary themselves and so our generalisation serves to emphasize the extent to which they are elementary and to which these properties are intrinsic to the definition of $\BMO{}{}$. In other cases, the known results follow from deeper theory and we are able to provide elementary proofs. We also prove many properties of $\BMO{}{}$ functions that may be well known, and may even be referred to in the literature, but for which we could not find a proof written down. An example of such a result is the completeness of $\BMO{}{}$, which is often deduced as a consequence of duality, or proven only for cubes in $\Rn$. We prove this result (Theorem~\ref{thm-complete}) for a general basis of shapes on a domain.

The paper has two primary focuses, the first being constants in inequalities related to $\BMO{}{}$. Considerable attention will be given to their dependence on an integrability parameter $p$, the basis of shapes used to define $\BMO{}{}$, and the dimension of the ambient Euclidean space. References to known results concerning sharp constants are given and connections between the sharp constants of various inequalities are established.   We distinguish between shapewise inequalities,  that is, inequalities that hold on any given shape, and norm inequalities.  We provide some elementary proofs of several shapewise inequalities and obtain sharp constants in the distinguished cases $p = 1$ and $p = 2$.  An example of such a result is the bound on truncations of a  $\BMO{}{}$ function (Proposition~\ref{pr:15}).  Although sharp shapewise inequalities are available for estimating the mean oscillation of the absolute value of a function in $\BMO{}{}$, the constant $2$ in the implied norm inequality - a statement of the boundedness of the map $f\mapsto|f|$ - is not sharp.  Rearrangements are a valuable tool that compensate for this, and we survey some known deep results giving norm bounds for decreasing rearrangements.

A second focus of this paper is on the product nature that $\BMO{}{}$ spaces may inherit from the shapes that define them. In the case where the shapes defining $\BMO{}{}$ have a certain product structure, namely that the collection of shapes coincides with the collection of Cartesian products of lower-dimensional shapes, a product structure is shown to be inherited by $\BMO{}{}$ under a mild hypothesis related to the theory of differentiation (Theorem~\ref{th:4}). This is particularly applicable to the case of strong $\BMO{}{}$.  It is important to note that the product nature studied here is different from that considered in the study of the space known as product $\BMO{}{}$ (see \cite{cf1,cf2}).

Following the preliminaries, Section 3 presents the basic theory of BMO on shapes. Section 4 concerns shapewise inequalities and the corresponding sharp constants. In Section 5, two rearrangement operators are defined and their boundedness on various function spaces is examined, with emphasis on $\BMO{}{}$. Section 6 discusses truncations of $\BMO{}{}$ functions and the cases where sharp inequalities can be obtained without the need to appeal to rearrangements. Section 7 gives a short survey of the John-Nirenberg inequality. Finally, in Section 8 we state and prove the product decomposition of certain $\BMO{}{}$ spaces.

This introduction is not meant as a review of the literature since that is part of the content of the paper, and references are given throughout the different sections.  The bibliography is by no means exhaustive, containing only a selection of the available literature, but it is collected with the hope of providing the reader with some standard or important references to the different topics touched upon here. 
 
\section{{\bf Preliminaries}}

Consider $\Rn$ with the Euclidean topology and Lebesgue measure, denoted by $|\cdot|$. By a domain we mean an open and connected set.
\begin{definition}
We call a shape in $\Rn$ any open set $S$ such that $0<|S|<\infty$. For a given domain $\Omega\subset \Rn$, we call a basis of shapes in $\Omega$ a collection $\cS$ of shapes $S$ such that $S\subset\Omega$ for all $S\in\cS$ and $\cS$ forms a cover of $\Omega$.
\end{definition}
Common examples of bases are the collections of all Euclidean balls, $\cB$, all cubes with sides parallel to the axes, $\cQ$, and all rectangles with sides parallel to the axes, $\cR$. In one dimension, these three choices degenerate to the collection of all (finite) open intervals, $\cI$. A variant of $\cB$ is $\cC$, the basis of all balls centered around some central point (usually the origin). Another commonly used collection is $\cQ_d$, the collection of all dyadic cubes, but
the open dyadic cubes cannot cover $\Omega$ unless $\Omega$ itself is a dyadic cube, so the proofs of some of the results below which rely on $\cS$ being an open cover (e.g.\ Proposition~\ref{pr:14} and Theorem~\ref{thm-complete})
may not apply.

One may speak about shapes that are balls with respect to a (quasi-)norm on $\Rn$, such as the $p$-``norms" $\norm{\cdot}{p}$ for $0<{p}\leq\infty$ when $n\geq{2}$. The case $p=2$ coincides with the basis $\cB$ and the case $p=\infty$ coincides with the basis $\cQ$, but other values of $p$ yield other interesting shapes. On the other hand, $\cR$ is not generated from a $p$-norm. 

Further examples of interesting bases have been studied in relation to the theory of differentiation of the integral, such as the collection of all rectangles with $j$ of the sidelengths being equal and the other $n-j$ being arbitrary (\cite{zy}), as well as the basis of all rectangles with sides parallel to the axes and sidelengths of the form $\big(\ell_1,\ell_2,\ldots,\phi(\ell_1,\ell_2,\ldots,\ell_{n-1})\big)$, where $\phi$ is a positive function that is monotone increasing in each variable separately (\cite{clm}).

\begin{definition}
Given two bases of shapes, $\cS$ and $\tilde{\cS}$, we say that $\cS$ is comparable to $\tilde{\cS}$, written $\cS\trianglelefteq\tilde{\cS}$, if there exist lower and upper comparability constants $c>0$ and $C>0$, depending only on $n$, such that for all $S\in\cS$ there exist $S_1,S_2\in\tilde{\cS}$ for which $S_1\subset S \subset S_2$ and $c|S_2|\leq|S|\leq C|S_1|$. If $\cS\trianglelefteq\tilde{\cS}$ and $\tilde{\cS}\trianglelefteq{\cS}$, then we say that $\cS$ and $\tilde{\cS}$ are equivalent, and write $\cS\approx\tilde{\cS}$. 
\end{definition}

An example of equivalent bases are $\cB$ and $\cQ$: one finds that $\cB\trianglelefteq\cQ$ with $c=\frac{\omega_n}{2^n}$ and $C=\omega_n\left(\frac{\sqrt{n}}{2}\right)^n$, and $\cQ\trianglelefteq\cB$ with $c=\frac{1}{\omega_n}\left(\frac{2}{\sqrt{n}}\right)^n$ and $C=\frac{2^n}{\omega_n}$, where $\omega_n$ is the volume of the unit ball in $\Rn$, and so $\cB\approx\cQ$.  The bases of shapes given by the balls in the other $p$-norms $\norm{\cdot}{p}$ for $1\leq{p}\leq\infty$ are also equivalent to these.

If $\cS\subset\tilde{\cS}$ then $\cS\trianglelefteq\tilde{\cS}$ with $c=C=1$. In particular, $\cQ\subset{\cR}$ and so $\cQ\trianglelefteq{\cR}$, but $\cR\ntrianglelefteq{\cQ}$ and so $\cQ\not\approx\cR$. 

Unless otherwise specified, we maintain the convention that $1\leq{p}<\infty$.  Moreover, many of the results implicitly assume that the functions are real-valued, but others
may hold also for complex-valued functions.  This should be understood from the context.

\section{{\bf $\bBMO{}{}$ spaces with respect to shapes}}

Consider a basis of shapes $\cS$. Given a shape $S\in\cS$, for a function $f\in \Lone(S)$, denote by $f_S$ its mean over $S$.
\begin{definition}
We say that a function satisfying $f\in \Lone(S)$ for all shapes $S\in\cS$ is in the space $\BMO{\cS}{p}(\Omega)$ if there exists a constant $K\geq 0$ such that
\begin{equation}\label{mosc}
\left(\dashint_{S}\!|f-f_S|^p\right)^{1/p}\leq{K},
\end{equation}
holds for all $S\in\cS$. 
\end{definition}
The quantity on the left-hand side of (\ref{mosc}) is called the $p$-mean oscillation of $f$ on $S$. For $f\in\BMO{\cS}{p}(\Omega)$, we define  $\norm{f}{\BMO{\cS}{p}}$ as the infimum of all $K$ for which (\ref{mosc}) holds for all $S\in\cS$. Note that the $p$-mean oscillation does not change if a constant is added to $f$; as such, it is sometimes useful to assume that a function has mean zero on a given shape. 

In the case where $p=1$, we will write $\BMO{\cS}{1}(\Omega)=\BMO{\cS}{}(\Omega)$.  For the classical $\BMO{}{}$ spaces we reserve the notation $\BMO{}{p}(\Omega)$ without explicit reference to the underlying basis of shapes ($\cQ$ or $\cB$).

We mention a partial answer to how $\BMO{\cS}{p}(\Omega)$ relate for different values of $p$. This question will be taken up in a later section when some more machinery has been developed. 

\begin{proposition}
\label{pr:1}
For any basis of shapes, $\BMO{\cS}{p_2}(\Omega)\subset \BMO{\cS}{p_1}(\Omega)$ with $\norm{f}{\BMO{\cS}{p_1}}\leq \norm{f}{\BMO{\cS}{p_2}}$ for $1\leq p_1\leq{p_2}<\infty$. In particular, this implies that $\BMO{\cS}{p}(\Omega)\subset \BMO{\cS}{}(\Omega)$ for all $1\leq p<\infty$.
\end{proposition}

\begin{proof}
This follows from Jensen's inequality with $p=\frac{p_2}{p_1}\geq{1}$.
\end{proof}

Next we show a lemma that implies, in particular, the local integrability of functions in $\BMO{\cS}{p}(\Omega)$. 

\begin{lemma}
\label{pr:18}
For any basis of shapes $\cS$, $\BMO{\cS}{p}(\Omega)\subset \Lploc(\Omega)$. 
\end{lemma}

\begin{proof}
Fix a shape $S\in\cS$ and a function $f\in \BMO{\cS}{p}(\Omega)$. By Minkowski's inequality on $\Lp(S,\frac{d{x}}{|S|})$,
$$
\left(\dashint_{S}\!|f|^p\right)^{1/p}\leq \left(\dashint_{S}\!|f-f_S|^p\right)^{1/p} +|f_S|
$$
and so
$$
\left(\int_{S}\!|f|^p\right)^{1/p}\leq |S|^{1/p}\left(\norm{f}{\BMO{\cS}{p}} +|f_S|\right).
$$

As $\cS$ covers all of $\Omega$, for any compact set $K\subset{\Omega}$ there exists a collection $\{S_i \}_{i=1}^{N}\subset\cS$ for some finite $N$ such that $K\subset\bigcup_{i=1}^{N}S_i$. Hence, using the previous calculation,
$$
\left(\int_{K}\!|f|^p\right)^{1/p}\leq\sum_{i=1}^{N}\left(\int_{S_i}\!|f|^p\right)^{1/p}\leq \sum_{i=1}^{N}|S_i|^{1/p}\left(\norm{f}{\BMO{\cS}{p}} +|f_{S_i}|\right)<\infty.
$$ 
\end{proof}

In spite of this, a function in $\BMO{\cS}{p}(\Omega)$ need not be locally bounded. If $\Omega$ contains the origin or is unbounded, $f(x)=\log|x|$ is the standard example of a function in $\BMO{}{}(\Omega) \setminus \Linfty(\Omega)$. The reverse inclusion, however, does hold:

\begin{proposition}
\label{pr:8}
For any basis of shapes, $\Linfty(\Omega)\subset \BMO{\cS}{p}(\Omega)$ with 
$$
\norm{f}{\BMO{\cS}{p}}\leq
\begin{cases}
\norm{f}{\Linfty}, & 1\leq{p}\leq{2};\\
2\norm{f}{\Linfty}, & 2<{p}<\infty.
\end{cases}
$$
\end{proposition}

\begin{proof}
Fix $f\in \Linfty(\Omega)$ and a shape $S\in\cS$. For any $1\leq{p}<\infty$, Minkowski's and Jensen's inequalities give 
$$
\left(\dashint_{S}\!|f-f_S|^p \right)^{1/p}\leq 2\left(\dashint_{S}\!|f|^p \right)^{1/p}\leq 2\norm{f}{\infty}.
$$

Restricting to $1\leq{p}\leq{2}$, one may use Proposition \ref{pr:1} with $p_2=2$ to arrive at
$$
\left(\dashint_{S}\!|f-f_{S}|^p \right)^{1/p}\leq\left(\dashint_{S}\!|f-f_{S}|^2 \right)^{1/2}.
$$
Making use of the Hilbert space structure on $L^2(S,\frac{d{x}}{|S|})$, observe that $f-f_S$ is orthogonal to constants and so it follows that 
$$
\dashint_{S}|f-f_S|^2 = \dashint_{S}\!|f|^2-|f_S|^2\leq\dashint_{S}\!|f|^2\leq \norm{f}{\infty}^2.
$$
\end{proof}

A simple example shows that the constant 1 obtained for $1\leq{p}\leq{2}$ is, in fact, sharp:

\begin{example}
\label{ex-simple}
Let $S$ be a shape on $\Omega$ and consider a function $f=\chi_{E}-\chi_{E^c}$, where $E$ is a measurable subset of $S$ such that $|E|=\frac{1}{2}|S|$ and $E^c=S\setminus E$. Then $f_S=0$,  $|f-f_{S}| = |f| \equiv 1$ on $S$ and so 
$$
\dashint_{S}\!|f-f_{S}|^p=1.
$$
Thus, $\norm{f}{\BMO{\cS}{p}}\geq{1}=\norm{f}{\Linfty}$.
\end{example}

There is no reason to believe that the constant 2 for $2<{p}<\infty$ is sharp, however, and so we pose the following question:

\begin{problem}
What is the smallest constant $c_{\infty}(p,\cS)$ such that the inequality $\norm{f}{\BMO{\cS}{p}}\leq c_{\infty}(p,\cS) \norm{f}{\Linfty}$ holds for all $f\in \BMO{\cS}{p}(\Omega)$?
\end{problem}

The solution to this problem was obtained by Leonchik in the case when $\Omega\subset\R$ and $\cS=\cI$.

\begin{theorem}[\cite{leo,ko2}]
$$
c_{\infty}(p,\cI)=2\sup_{0<h<1}\{h(1-h)^p+h^p(1-h)\}^{1/p}.
$$
\end{theorem}

An analysis of this expression (\cite{ko2}) shows that $c_{\infty}(p,\cI)=1$ for $1\leq{p}\leq{3}$, improving on Proposition \ref{pr:8} for $2<p\leq{3}$. Moreover, $c_{\infty}(p,\cI)$ is monotone in $p$ with $1<c_{\infty}(p,\cI)<2$ for $p>3$ and $c_{\infty}(p,\cI)\rightarrow{2}$ as $p\rightarrow{\infty}$.

It is easy to see that $\norm{\cdot}{\BMO{\cS}{p}}$ defines a seminorm. It cannot be a norm, however, as a function $f$ that is almost everywhere equal to a constant will satisfy $\norm{f}{\BMO{\cS}{p}}=0$. What we can show is that the quantity $\norm{\cdot}{\BMO{\cS}{p}}$ defines a norm modulo constants. 

In the classical case of $\BMO{}{}(\Rn)$, the proof (see \cite{gra}) relies on the fact that $\cB$ contains $\cC$ and so $\Rn$ may be written as the union of countably-many concentric shapes. When on a domain that is also a shape, the proof is immediate. In our general setting, however, we may not be in a situation where $\Omega$ is a shape or $\cS$ contains a distinguished subcollection of nested shapes that exhausts all of $\Omega$; for an example, consider the case where $\Omega$ is a rectangle that is not a cube with $\cS=\cQ$. As such, the proof must be adapted. We do so in a way that relies on shapes being open sets that cover $\Omega$, and on $\Omega$ being connected and Lindel\"{o}f.

\begin{proposition} 
\label{pr:14}
For any basis of shapes, $\norm{f}{\BMO{\cS}{p}}=0$ if and only if $f$ is almost everywhere equal to some constant.
\end{proposition}

\begin{proof}
One direction is immediate. For the other direction, fix $f\in \BMO{\cS}{p}(\Omega)$ such that $\norm{f}{\BMO{\cS}{p}}=0$. It follows that $f=f_S$ almost everywhere on $S$ for each $S\in\cS$. 

Fix $S_0\in\cS$ and let $C_0 = f_{S_0}$.  Set
$$
U=\bigcup\{S\in\cS:f_S=C_0 \}.
$$ 
Using the Lindel\"{o}f property of $\Omega$, we may assume that $U$ is defined by a countable union. It follows that $f=C_0$ almost everywhere on $U$ since $f=f_S=C_0$ almost everywhere for each $S$ comprising $U$. The goal now is to show that $\Omega=U$.

Now, let 
$$
V=\bigcup\{S\in\cS:f_{S}\neq C_0 \}.
$$ 
Since $\cS$ covers $\Omega$, we have $\Omega=U \cup V$. Thus, in order to show that $\Omega=U$, we need to show that $V$ is empty. To do this, we note that both $U$ and $V$ are open sets and so, since $\Omega$ is connected, it suffices to show that $U$ and $V$ are disjoint.

Suppose that $x\in U\cap V$. Then there is an $S_1$ containing $x$ which is in $U$ and an $S_2$ containing $x$ which is in $V$. In particular, $f=C_0$ almost everywhere on $S_1$ and $f\neq{C_0}$ almost everywhere on $S_2$. However, this is impossible as $S_1$ and $S_2$ are open sets of positive measure with non-empty intersection and so $S_1\cap S_2$ must have positive measure. Therefore, $U\cap{V}=\emptyset$ and the result follows. 
\end{proof}

As Proposition \ref{pr:14} implies that $\BMO{\cS}{p}(\Omega)/\C$ is a normed linear space, a natural question is whether it is complete. For $\BMO{}{}(\Rn)/\C$ this is a corollary of Fefferman's theorem that identifies it as the dual of the real Hardy space (\cite{fe,fst}). 

A proof that $\BMO{}{}(\Rn)/\C$ is a Banach space that does not pass through duality came a few years later and is due to Neri (\cite{ne}). This is another example of a proof that relies on the fact that $\cB$ contains $\cC$ (or, equivalently, that $\cQ$ contains an analogue of $\cC$ but for cubes). The core idea, however, may be adapted to our more general setting. The proof below makes use of the fact that shapes are open and that $\Omega$ is path connected since it is both open and connected. 

\begin{theorem}
\label{thm-complete}
For any basis of shapes, $ \BMO{\cS}{p}(\Omega)$ is complete.
\end{theorem}

\begin{proof}
Let $\{f_i\}$ be Cauchy in $\BMO{\cS}{p}(\Omega)$. Then, for any shape $S\in\cS$, the sequence $\{f_i-(f_i)_S \}$ is Cauchy in $\Lp(S)$. Since $\Lp(S)$ is complete, there exists a function
$f^S\in \Lp(S)$ such that $f_i-(f_i)_S\rightarrow f^S$ in $\Lp(S)$. The function $f^S$ can be seen to have mean zero on $S$: since $f_i-(f_i)_S$ converges to $f^S$ in $\Lone(S)$, it follows that 
\begin{equation}
\label{lim}
\dashint_{S}\!f^S=\lim_{i\rightarrow\infty}\dashint_{S}f_i-(f_i)_S=0.
\end{equation}

If we have two shapes $S_1,S_2\in\cS$ such that $S_1\cap S_2\neq\emptyset$, by the above there is a function $f^{S_1}\in \Lp(S_1)$ such that $f_i-(f_i)_{S_1}\rightarrow f^{S_1}$ in $\Lp(S_1)$ and a function $f^{S_2}\in \Lp(S_2)$ such that $f_i-(f_i)_{S_2}\rightarrow f^{S_2}$ in $\Lp(S_2)$. Since both of these hold in $\Lp(S_1\cap S_2)$, we have
$$
(f_i)_{S_2}-(f_i)_{S_1}=[f_i-(f_i)_{S_1} ]- [f_i-(f_i)_{S_2} ]\rightarrow f^{S_1}-f^{S_2} \quad \mbox{in } \Lp(S_1\cap S_2).$$
This implies that the sequence $C_i(S_1,S_2)=(f_i)_{S_2}-(f_i)_{S_1}$  converges as constants to a limit that we denote by $C(S_1,S_2)$, with
$$f^{S_1}-f^{S_2} \equiv C(S_1,S_2)\quad \mbox{on }S_1\cap S_2.$$

From their definition, these constants are antisymmetric: 
$$C(S_1,S_2) = -C(S_2,S_1).$$
Moreover, they possess an additive property that will be useful in later computations.  By a finite chain of shapes we mean a finite sequence $\{S_j\}_{j=1}^{k}\subset\cS$ such that $S_j\cap S_{j+1}\neq\emptyset$ for all $1\leq j\leq{k-1}$. Furthermore, by a loop of shapes we mean a finite chain $\{S_j\}_{j=1}^{k}$ such that $S_1\cap S_{k}\neq\emptyset$.
If $\{S_j \}_{j=1}^{k}$ is a loop of shapes, then 
\begin{equation}
\label{add}
C(S_1,S_k)= \sum_{j=1}^{k-1} C(S_j,S_{j+1}).
\end{equation}
To see this, consider the telescoping sum
$$
(f_i)_{S_k}-(f_i)_{S_1}=\sum_{j=1}^{k-1}(f_i)_{S_{j+1}}-(f_i)_{S_j},
$$
for a fixed $i$. The formula (\ref{add}) follows from this as each $(f_i)_{S_{j+1}}-(f_i)_{S_j}$ converges to $C(S_j,S_{j+1})$ since $S_{j}\cap S_{j+1}\neq\emptyset$ and $(f_i)_{S_k}-(f_i)_{S_1}$ converges to $C(S_1,S_k)$ since $S_1\cap S_k\neq\emptyset$.

Let us now fix a shape $S_0\in\cS$ and consider another shape $S\in\cS$ such that $S_0\cap S=\emptyset$. Since $\Omega$ is a path-connected set, for any pair of points $(x,y)\in S_0\times S$ there exists a path $\gamma_{x,y}:[0,1]\rightarrow \Omega$ such that $\gamma_{x,y}(0)=x$ and $\gamma_{x,y}(1)=y$. Since $\cS$ covers $\Omega$ and the image of $\gamma_{x,y}$ is a compact set, we may cover $\gamma_{x,y}$ by a finite number of shapes. From this we may extract a finite chain connecting $S$ to $S_0$.

We now come to building the limit function $f$. If $x\in S_0$, then set $f(x)=f^{S_0}(x)$, where $f^{S_0}$ is as defined earlier. If $x\notin S_0$, then there is some shape $S$ containing $x$ and, by the preceding argument, a finite chain of shapes $\{S_j \}_{j=1}^{k}$ where $S_k=S$. In this case, set
\begin{equation}
\label{def}
f(x)=f^{S_k}(x)+\sum_{j=0}^{k-1} C(S_j,S_{j+1}).
\end{equation}
The first goal is to show that this is well defined. Let $\{\tilde{S}_j\}_{j=1}^{\ell}$ be another finite chain connecting some $\tilde{S}_\ell$ with $x \in \tilde{S}_\ell$ to $S_0=\tilde{S}_0$. Then we need to show that
\begin{equation}
\label{need}
f^{S_k}(x)+\sum_{j=0}^{k-1} C(S_j,S_{j+1})=f^{\tilde{S}_\ell}(x)+\sum_{j=0}^{\ell-1} C(\tilde{S}_j,\tilde{S}_{j+1}).
\end{equation}
First, we use the fact that $x \in S_k\cap \tilde{S}_\ell$ to write $ f^{S_k}(x)- f^{\tilde{S}_\ell}(x) = C(S_{k},\tilde{S}_\ell)$. Then, from the antisymmetry property of the constants, (\ref{need}) is equivalent to  
$$
C(S_{k},\tilde{S}_\ell)= C(S_k,S_{k-1})+\cdots+C(S_1,S_0)+C(S_0,\tilde{S}_{1})+\cdots+ C(\tilde{S}_{\ell-1},\tilde{S}_{\ell}), 
$$
which follows from (\ref{add}) as $\{S_{k},S_{k-1},\ldots,S_{1},S_{0},\tilde{S}_1,\tilde{S}_2,\ldots,\tilde{S}_{\ell}\}$ forms a loop.

Finally, we show that $f_i\rightarrow f$ in $\BMO{\cS}{p}(\Omega)$. Fixing a shape $S\in\cS$, choose a finite chain $\{S_j\}_{j=1}^{k}$ such that $S_k=S$. By (\ref{def}), on $S$ we have that $f=f^{S}$ modulo constants, and so, using (\ref{lim}) and the definition of $f^S$,
\[
\begin{split}
\dashint_{S}|(f_i(x)-f(x))-(f_i-f)_S|^p\,dx&=\dashint_{S}|(f_i(x)-f^S(x))-(f_i-f^S)_S|^p\,dx\\&=\dashint_{S}|f_i(x)-(f_i)_{S}-f^S(x)|^p\,dx\\&\rightarrow{0}\qquad\text{as $i\rightarrow\infty$}.
\end{split}
\]
\end{proof}

\section{{\bf Shapewise inequalities on $\bBMO{}{}$}}
A ``shapewise" inequality is an inequality that holds for each shape $S$ in a given basis. In this section, considerable attention will be given to highlighting those situations where the constants in these inequalities are known to be sharp. A recurring theme is the following: sharp results are mainly known for $p=1$ and for $p=2$ and, in fact, the situation for $p=2$ is usually simple. The examples given demonstrating sharpness are straightforward generalisations of some of those found in \cite{ko2}.

For this section, we assume that $\cS$ is an arbitrary basis of shapes and that $f\in \Lone(S)$ for every $S\in\cS$.

We begin by considering inequalities that provide equivalent characterizations of $\BMO{\cS}{p}(\Omega)$. As with the classical $\BMO{}{}$ space, one can estimate the mean oscillation of a function on a shape by a double integral that is often easier to use for calculations but that comes at the loss of a constant. 

\begin{proposition}
\label{pr:9}
For any shape $S\in \cS$, 
$$
\frac{1}{2}\left(\dashint_{S}\dashint_{S}\!|f(x)-f(y)|^p\,dy\,dx\right)^{\frac{1}{p}}\!\!\leq\left(\dashint_{S}\!|f-f_{S}|^p\right)^{\frac{1}{p}}\!\!\leq\left(\dashint_{S}\dashint_{S}\!|f(x)-f(y)|^p\,dy\,dx\right)^{\frac{1}{p}}\!\!.
$$
\end{proposition}

\begin{proof}
Fix a shape $S\in\cS$. By Jensen's inequality we have that 
$$
\dashint_{S}\!|f(x)-f_{S}|^p\,dx=\dashint_{S}\!\left|\dashint_{S}\!\big(f(x)- f(y)\big)\,dy\right|^p\,dx\leq \dashint_{S}\dashint_{S}\!|f(x)- f(y)|^p\,dy\,dx
$$
and by Minkowski's inequality on $\Lp(S\times S,\frac{d{x}d{y}}{|S|^2} )$ we have that
\[
\begin{split}
\left(\dashint_{S}\dashint_{S}\!|f(x)- f(y)|^p\,dy\,dx\right)^{1/p} &=\left(\dashint_{S}\dashint_{S}\!\left|\big(f(x)-f_S\big)- \big(f(y)-f_S\big)\right|^p\,dy\,dx\right)^{1/p}\\&\leq 2 \left( \dashint_{S}\!|f-f_S|^p\right)^{1/p}.
\end{split}
\]
\end{proof}

When $p=1$, the following examples show that the constants in this inequality are sharp.

\begin{example}
\label{ex:2}
Let $S$ be a shape in $\Omega$ and consider a function $f=\chi_{E}$, where $E$ is a measurable subset of $S$ such that $|E|=\frac{1}{2}|S|$. Then $f_S=\frac{1}{2}$ and so $|f-f_S|=\frac{1}{2}\chi_{S}$, yielding
$$
\dashint_{S}\!|f-f_{S}|=\frac{1}{2}.
$$
Writing $E^c=S\setminus{E}$, we have that $|f(x)-f(y)|=0$ for $(x,y)\in E\times E$ or $(x,y)\in E^c\times E^c$, and that $|f(x)-f(y)|=1$ for $(x,y)\in E^c\times E$; hence,
$$
\dashint_{S}\dashint_{S}\!|f(x)-f(y)|\,dy\,dx = \frac{2}{|S|^2}\int_{E^c}\int_{E}\!|f(x)-f(y)|\,dy\,dx=\frac{2|E^c||E|}{|S|^2}=\frac{1}{2}.
$$
Therefore, the right-hand side constant 1 is sharp. 
\end{example}

\begin{example}
\label{ex:3}
Now consider a function ${f}=\chi_{E_1}-\chi_{E_3}$, where $E_1,E_2,E_3$ are measurable subsets of $S$ such that $S=E_1\cup E_2\cup E_3$ is a disjoint union (up to a set of measure zero) and $|E_1|=|E_3|=\beta|S|$ for some $0<\beta<\frac{1}{2}$. Then, $f_S=0$ and so
$$
\dashint_{S}\!|f-f_{S}| =  \dashint_{S}\!|f|= 2\beta.
$$
Since $|f(x)-f(y)|=0$ for $(x,y)\in E_j\times E_j$, $j=1,2,3$, $|f(x)-f(y)|=1$ for $(x,y)\in E_i\times E_j$ when $|i-j|=1$, and $|f(x)-f(y)|=2$ for $(x,y)\in E_i\times E_j$ when $|i-j|=2$, we have that
$$
\dashint_{S}\dashint_{S}\!|f(x)-f(y)|\,dy\,dx=\frac{4\beta|S|(|S|-2\beta|S|)+4\beta^2|S|^2}{|S|^2}=4\beta(1-2\beta)+4\beta^2.
$$
As
$$
\frac{2\beta}{4\beta(1-2\beta)+4\beta^2} = \frac{1}{2-2\beta}\rightarrow\frac{1}{2}
$$
as $\beta\rightarrow{0^+}$, the left-hand side constant $\frac{1}{2}$ is sharp.
\end{example}

For other values of $p$, however, these examples tell us nothing about the sharpness of the constants in Proposition \ref{pr:9}. In fact, the constants are not sharp for $p=2$. In the following proposition, as in many to follow, the additional Hilbert space structure afforded to us yields a sharp statement (in this case, an equality) for little work.

\begin{proposition}
For any shape $S\in \cS$, 
$$
\left( \dashint_{S}\dashint_{S}\!|f(x)-f(y)|^2\,dy\,dx\right)^{1/2}=\sqrt{2}\left(\dashint_{S}\!|f-f_{S}|^2\right)^{1/2}.
$$
\end{proposition}

\begin{proof}
Observe that as elements of $L^2(S\times S,\frac{dx\,dy}{|S|^2})$, $f(x)-f_S$ is orthogonal to $f(y)-f_S$. Thus,
$$
\dashint_{S}\dashint_{S}\!|f(x)-f(y)|^2\,dy\,dx= \dashint_{S}\dashint_{S}\!|f(x)-f_S|^2\,dy\,dx+ \dashint_{S}\dashint_{S}\!|f(y)-f_S|^2\,dy\,dx
$$
and so 
$$
\left( \dashint_{S}\dashint_{S}\!|f(x)-f(y)|^2\,dy\,dx\right)^{1/2}=\sqrt{2}\left(\dashint_{S}\!|f-f_{S}|^2\right)^{1/2}.
$$
\end{proof}

In a different direction, it is sometimes easier to consider not the oscillation of a function from its mean, but its oscillation from a different constant. Again, this can be done at the loss of a constant.

\begin{proposition}
\label{pr:6}
For any shape $S\in \cS$,
$$
\inf_{c} \left(\dashint_{S}\!|f-c|^{p}\right)^{1/p} \leq \left(\dashint_{S}\!|f-f_{S}|^{p}\right)^{1/p}\leq 2 \inf_{c}\left(\dashint_{S}\!|f-c|^{p}\right)^{1/p},
$$
where the infimum is taken over all constants $c$.
\end{proposition}

\begin{proof} 
The first inequality is trivial. To show the second inequality, fix a shape $S\in\cS$. By Minkowski's inequality on $\Lp(S,\frac{d{x}}{|S|})$ and Jensen's inequality,
$$
\left(\dashint_{S}\!|f-f_{S}|^{p}\right)^{1/p}\leq \left(\dashint_{S}\!|f-c|^{p}\right)^{1/p}+\left(|f_{S}-c|^p\right)^{1/p}\leq 2\left(\dashint_{S}\!|f-c|^{p}\right)^{1/p}.
$$
\end{proof}

As with Proposition \ref{pr:9}, simple examples show that the constants are sharp for $p=1$. 

\begin{example}
Let $S$ be a shape on $\Omega$ and consider the function $f=\chi_{E}-\chi_{E^c}$ as in Example~\ref{ex-simple}, with $|E| = |S|/2$, $E^c = S\setminus E$.  Then
$$
\dashint_{S}\!|f-f_S|=1.
$$
and for any constant $c$  we have that
$$
\dashint_{S}\!|f-c|=\frac{|1-c||E|+|1+c||E^c|}{|S|} =\frac{|1-c|+|1+c|}{2} \geq{1}
$$
with equality if $c\in[-1,1]$.  Thus
$$
\inf_{c}\dashint_{S}\!|f-c| =\dashint_{S}\!|f-f_{S}|,
$$
showing the left-hand side constant $1$ in Proposition \ref{pr:6} is sharp when $p=1$. 
\end{example}

\begin{example}
Consider, now, the function $f=\chi_{E}$ where ${E}$ is a measurable subset of $S$ such that $|E|=\alpha|S|$ for some $0<\alpha<\frac{1}{2}$. Then, $f_S=\alpha$ and 
$$
\dashint_{S}\!|f-f_S|=\frac{(1-\alpha)\alpha|S|}{|S|}+\frac{(|S|-\alpha|S|)\alpha}{|S|}=2\alpha(1-\alpha).
$$
For any constant $c$, we have that
$$
\dashint_{S}\!|f-c|=\frac{\big|1-|c|\big||E|}{|S|}+\frac{|c|(|S|-|E|)}{|S|}=\alpha\big|1-|c|\big|+|c|(1-\alpha).
$$
The right-hand side is at least $\alpha(1-|c|)+|c|(1-\alpha)=\alpha+|c|(1-2\alpha)$, which is at least $\alpha$ with equality for $c=0$, and so
$$
\inf_{c}\dashint_{S}\!|f-c|=\alpha.
$$
As $\alpha\rightarrow{0^+}$, 
$$
\frac{2\alpha(1-\alpha)}{\alpha}\rightarrow{2},
$$
showing the right-hand side constant 2 in Proposition \ref{pr:6} is sharp when $p=1$. 
\end{example}

When $p=1$, it turns out that we know for which constant the infimum in Proposition \ref{pr:6} is achieved. 

\begin{proposition}
Let $f$ be real-valued. For any shape $S\in\cS$,
$$
\inf_{c} \dashint_{S}\!|f-c|=\dashint_{S}\!|f-m|,
$$
where $m$ is a median of $f$ on $S$: that is, a (possibly non-unique) number such that $|\{x\in{S}:f(x)>m\}|\leq\frac{1}{2}|S|$ and $|\{x\in{S}:f(x)<m\}|\leq\frac{1}{2}|S|$.
\end{proposition}

Note that the definition of a median makes sense for real-valued measurable functions. A proof of this proposition can be found in the appendix of \cite{css}, along with the fact that such functions always have a median on a measurable set of positive and finite measure (in particular, on a shape). Also, note that from the definition of a median, it follows that 
$$
|\{x\in{S}:f(x)\geq{m} \}|=|S|-|\{x\in{S}:f(x)>{m} \}|\geq|S|-\frac{1}{2}|S|=\frac{1}{2}|S|
$$ 
and, likewise, 
$$
|\{x\in{S}:f(x)\leq{m} \}|\geq\frac{1}{2}|S|.
$$

\begin{proof}
Fix a shape $S\in\cS$ and a median $m$ of $f$ on $S$. For any constant $c$,
\[
\begin{split}
\int_{S}\!|f(x)-m|\,dx&=\int_{\{x\in{S}:f(x)\geq{m} \}}\!\big(f(x)-m\big)\,dx+\int_{\{x\in{S}:f(x)<{m} \}}\!\big(m-f(x)\big)\,dx\\ &=\int_{\{x\in{S}:f(x)\geq{m} \}}\!\big(f(x)-c\big)\,dx+(c-m)|\{x\in{S}:f(x)\geq{m} \}| \\ 
&\qquad+\int_{\{x\in{S}:f(x)<{m} \}}\!\big(c-f(x)\big)\,dx+(m-c)|\{x\in{S}:f(x)<{m} \}|.
\end{split}
\]

Assuming that $m>c$, we have that 
$$
\int_{\{x\in{S}:f(x)\geq{m} \}}\!\big(f(x)-c\big)\,dx\leq \int_{\{x\in{S}:f(x)\geq{c} \}}\!\big(f(x)-c\big)\,dx
$$
and
$$
\int_{\{x\in{S}:c\leq f(x)<{m} \}}\!\big(c-f(x)\big)\,dx\leq{0},
$$
and so we can write
\[
\begin{split}
\int_{S}\!|f(x)-m|\,dx&\leq \int_{\{x\in{S}:f(x)\geq{c} \}}\!\big(f(x)-c\big)\,dx+(c-m)|\{x\in{S}:f(x)\geq{m} \}| \\ 
&\qquad+\int_{\{x\in{S}:f(x)<{c} \}}\!\big(c-f(x)\big)\,dx+(m-c)|\{x\in{S}:f(x)<{m} \}|\\&=\int_{S}\!|f(x)-c|\,dx\\&\qquad+(m-c)\big[|\{x\in{S}:f(x)<{m} \}|-|\{x\in{S}:f(x)\geq{m} \}|\big]\\&\leq \int_{S}\!|f(x)-c|\,dx+(m-c)\left[\frac{1}{2}|S|-\frac{1}{2}|S|\right]\\&=\int_{S}\!|f(x)-c|\,dx.
\end{split}
\]

In the case where $m<c$, we may apply the previous calculation to $-f$ and use the fact that $-m$ is a median of $-f$:
$$
\int_{S}\!|f(x)-m|\,dx=\int_{S}\!|-f(x)-(-m)|\,dx\leq \int_{S}\!|-f(x)-(-c)|\,dx=\int_{S}\!|f(x)-c|\,dx.
$$
\end{proof}

For $p=2$, we are able to do a few things at once. We are very simply able to obtain an equality that automatically determines the sharp constants for Proposition \ref{pr:6} and the constant $c$ for which the infimum is achieved. 

\begin{proposition}
\label{pr:12}
For any shape $S\in\cS$,
$$
\inf_{c} \left(\dashint_{S}\!|f-c|^{2}\right)^{1/2}=\left(\dashint_{S}\!|f-f_S|^{2}\right)^{1/2},
$$
where the infimum is taken over all constants $c$.
\end{proposition}

\begin{proof}
Fix a shape $S\in\cS$ and a constant $c$. As previously observed, $f-f_S$ is orthogonal to any constant in the sense of $L^2(S,\frac{dx}{|S|})$; in particular, $f-f_S$ is orthogonal to $f_S-c$. Thus, 
$$
\dashint_{S}\!|f-c|^{2}=\dashint_{S}\!|f-f_S|^{2}+\dashint_{S}\!|f_S-c|^{2}\geq \dashint_{S}\!|f-f_S|^{2}
$$
with the minimum achieved when $c=f_S$.
\end{proof}

The following proposition shows that the action of H\"{o}lder continuous maps preserves the bound on the $p$-mean oscillation, up to a constant.

\begin{proposition}
\label{pr:2}
Let $F:\R\rightarrow\R$ be $\alpha$-H\"{o}lder continuous for $0<\alpha\leq{1}$ with H\"{o}lder coefficient $L$. Fix any shape $S\in\cS$ and
suppose $f \in \Lone(S)$ is real-valued.  Then, for $1 \leq p < \infty$,
\begin{equation}
\label{eqn-Holderp}
\left(\dashint_{S}\!\big|F\circ f-(F\circ f)_{S}\big|^p\right)^{\frac{1}{p}}\leq 2 L\left(\dashint_{S}\!|f-f_S|^{p}\right)^{\frac{\alpha}{p}}.
\end{equation}
When $p=2$, 
\begin{equation}
\label{eqn-Holder2}
\left(\dashint_{S}\!\big|F\circ f-(F\circ f)_{S}\big|^2\right)^{\frac{1}{2}}\leq  L\left(\dashint_{S}\!|f-f_S|^{2}\right)^{\frac{\alpha}{2}}.
\end{equation}
\end{proposition}

\begin{proof}
Fix a shape  $S\in\cS$. By Proposition \ref{pr:6} and Jensen's inequality, we have that
\[
\begin{split}
\dashint_{S}\!\big|F\circ f-(F\circ f)_{S}\big|^p&\leq 2^p\dashint_{S}\!\big|F(f(x))-F( f_S)\big|^p dx
\\ &\leq  2\Lp \dashint_{S}\!|f(x)-f_S|^{\alpha p}\,dx\\&\leq 2^p \Lp\left(\dashint_{S}\!|f-f_S|^{p}\right)^{\alpha}.
\end{split}
\]
When $p=2$, Proposition~\ref{pr:12} shows that the factor of $2^p$ in the first inequality can be dropped.
\end{proof}

As has been pointed out in \cite{css}, if one uses the equivalent norm defined by Proposition \ref{pr:6}, the result of Proposition \ref{pr:2} holds with constant $L$ for any $p \geq 1$,
since the factor of $2$ comes from $(F\circ f)_{S}\neq F(f_{S})$.

The following example demonstrates that the constants are sharp when $p=1$ and $p = 2$.

\begin{example} 
\label{ex-absval}
Consider the function $F(x)=|x|$, so that $\alpha = 1=L$,
and fix a shape $S$ in $\Omega$. Taking the function ${f}=\chi_{E_1}-\chi_{E_3}$, as in Example \ref{ex:3}, where $S$ is a disjoint union $E_1\cup E_2\cup E_3$ and $|E_1|=|E_3|=\beta|S|$ for some $0<\beta<\frac{1}{2}$, we have
 $|f|_S=2\beta$ and so 
$$
\dashint_{S}\!\big||f|-|f|_{S}\big|=4(1-2\beta)\beta.
$$
Since $\dashint_{S}\!|f-f_{S}| =  2\beta$
and
$
\frac{4(1-2\beta)\beta}{2\beta}\rightarrow{2}
$
as $\beta\rightarrow{0^+}$,  the constant $2$ is sharp for $p = 1$.

For $p = 2$, when $f\geq{0}$ we have $F(f) = f$ and $F(f)_{S}= F(f_S)$, so equality holds in \eqref{eqn-Holder2}.
\end{example}

While we have shown that the shapewise inequalities \eqref{eqn-Holderp}, for $p=1$, and \eqref{eqn-Holder2} are sharp for $F(x)=|x|$, in the next section it will be shown that better constants can be obtained for norm inequalities. 

Now we address how the $\BMO{{\cS}}{p}(\Omega)$ spaces relate for different bases.

\begin{proposition}
\label{pr:5}
For any shape $S\in\cS$, if $\tilde{S}$ is another shape (from possibly another basis) such that $\tilde{S}\subset{S}$ and $|\tilde{S}|\geq c|{S}|$ for some constant $c$, then
$$
\left(\dashint_{\tilde{S}}\!|f-f_{\tilde{S}}|^p\right)^{\frac{1}{p}}\leq 2c^{-1/p}\left(\dashint_{{S}}\!|f-f_{{S}}|^p\right)^{\frac{1}{p}}.
$$
\end{proposition}

\begin{proof}
From Proposition \ref{pr:6},
$$
\dashint_{\tilde{S}}\!|f-f_{\tilde{S}}|^p\leq 2^p\dashint_{\tilde{S}}\!|f-f_{S}|^p\leq 2^pc^{-1}\dashint_{S}\!|f-f_{S}|^p.
$$
\end{proof}

An immediate consequence of this is that $\BMO{\cS}{p}(\Omega)\subset\BMO{\tilde{\cS}}{p}(\Omega)$ if $\tilde{\cS}\trianglelefteq\cS$. Moreover, if $\cS\approx\tilde{\cS}$ then $\BMO{\cS}{p}(\Omega)\cong \BMO{\tilde{\cS}}{p}(\Omega)$. In particular, it follows that $\BMO{\cB}{p}(\Omega)\cong \BMO{\cQ}{p}(\Omega)$. Since $\cQ\subset\cR$, it is automatic without passing through Proposition \ref{pr:5} that $\BMO{\cR}{p}(\Omega)\subset \BMO{}{p}(\Omega)$ with $\norm{f}{\BMO{}{p}}\leq\norm{f}{\BMO{\cR}{p}}$. The reverse inclusion is false. The following example of a function in $\BMO{}{}(\Omega)$ that is not in $\BMO{\cR}{}(\Omega)$ is taken from \cite{ko2}, where the calculations proving the claim can be found: 

\begin{example}
\label{ex:1}
Consider $\Omega=(0,1)\times(0,1)\subset\R^2$. The function 
$$
f=\sum_{k=1}^{\infty}\chi_{\left(0,{2^{-k+1}}\right)\times\left(0,\frac{1}{k}\right) }
$$
belongs to $\BMO{}{}(\Omega) \setminus \BMO{\cR}{}(\Omega)$.
\end{example}

\section{{\bf Rearrangements and the absolute value}}

Consider two measure spaces $(M,\mu)$ and $(N,\nu)$ such that $\mu(M)=\nu(N)$.
\begin{definition}
\label{def-equimeas}
We say that measurable functions $f:M\rightarrow\R$ and $g:N\rightarrow\R$ are equimeasurable if for all $s\in\R$ the quantities $\mu_f(s)=\mu\big(\{x \in M:f(x)>s \}\big)$ and $\nu_g(s)=\nu\big(\{y\in N:g(y)>s \}\big)$ coincide. 
\end{definition}

It is important to note that this is not the standard definition of equimeasurability. Typically (see, for example, \cite{bs}) equimeasurability means $\mu_{|f|}(s)=\mu_{|g|}(s)$ for all $s\geq{0}$; however, for our purposes, it will be useful to distinguish between two functions being equimeasurable and the absolute value of two functions being equimeasurable. That said, it is true that

\begin{lemma}
\label{lm:1}
Let $f$ and $g$ be measurable functions such that $\mu_f(s)=\nu_g(s)<\infty$ for all $s$. Then, $\mu_{|f|}(s)=\nu_{|g|}(s)$ for all $s$. 
\end{lemma}

\begin{proof}
Fix $s\in\R$. Writing
$$
\{x \in M:f(x)<-s \}=\bigcup_{n\in\N} \left\{x\in M:f(x)\leq-s-\frac{1}{n}\right\},
$$
we have that 
$$
\mu\big(\{x \in M:f(x)<-s \}\big)=\lim_{n\rightarrow\infty} \mu(M)-\mu_{f}\left(-s-\frac{1}{n} \right).
$$
Here we use the convention that infinity minus a finite number is infinity and use the fact that $\mu_{f}<\infty$. Since $\mu_f=\nu_g$, by assumption, it follows that 
$$
\mu\big(\{x \in M:f(x)<-s \}\big)=\nu\big(\{x \in N:g(x)<-s \}\big).
$$
If $s\geq{0}$, then
\[
\begin{split}
\mu\big(\{x \in M:|f(x)|>s \}\big) &= \mu\big(\{x \in M:f(x)>s \}\big)+\mu\big(\{x \in M:f(x)<-s \}\big)\\&=\nu\big(\{y \in N:g(y)>s \}\big)+\nu\big(\{y \in N:g(y)<-s \}\big)\\&=\nu\big(\{y \in N:|g(x)|>s \}\big).
\end{split}
\]
If $s<0$, then 
$$
\mu\big(\{x \in M:|f(x)|>s \}\big)=\mu(M)=\nu(N)=\nu\big(\{y \in N:|g(x)|>s \}\big).
$$
\end{proof}

A useful tool is the following lemma. It is a consequence of Cavalieri's principle, also called the layer cake representation, which provides a way of expressing the integral of $\varphi(|f|)$ for a suitable transformation $\varphi$ in terms of a weighted integral of $\mu_{|f|}$. The simplest incarnation of this principal states that for any measurable set $A$,
$$
\int_{A}\!|f|^p=\int_{0}^{\infty}\!p\alpha^{p-1}|\{x\in A:|f(x)|>\alpha \}|\,d\alpha,
$$
where $0<p<\infty$. A more general statement can be found in \cite{ll}, Theorem 1.13 and its remarks.

\begin{lemma}
\label{lm:3}
Let $M\subset\Rm$, $N\subset\Rn$ be Lebesgue measurable sets of equal measure, and $f:M \ra \R$ and $g: N \ra \R$ be measurable functions such that $|f|$ and $|g|$ are equimeasurable. Then, for $0<p<\infty$,
$$
\int_{M}\!|f|^p =\int_{N}\!|g|^p \qquad \text{and} \qquad \esssup_M |f| = \esssup_N |g|.
$$

Furthermore, under the hypothesis of Lemma \ref{lm:1}, for any constant $c$,
$$
\int_{M}\!\big||f|-c\big|=\int_{N}\!\big||g|-c\big|.
$$
\end{lemma}

Moving back to the setting of this paper, for this section we assume that $f$ is a measurable function on $\Omega$ that satisfies the condition 
\begin{equation}
\label{condition}
|\{x\in \Omega:|f(x)|>s \}|\rightarrow{0} \qquad\text{as} \qquad s\rightarrow\infty.
\end{equation}
This guarantees that the rearrangements defined below are finite on their domains (see \cite{sw}, V.3). 

\begin{definition}
Let $I_\Omega=(0,|\Omega|)$.
The decreasing rearrangement of $f$ is the function 
$$
f^\ast(t)=\inf\{s\geq{0}:|\{x\in \Omega:|f(x)|>s \}|\leq{t}\}, \quad t\in I_\Omega.
$$
\end{definition}

This rearrangement is studied in the theory of interpolation and rearrangement-invariant function spaces. In particular, it can be used to define the Lorentz spaces, $L^{p,q}$, which are a refinement of the scale of Lebesgue spaces and can be used to strengthen certain inequalities such as those of Hardy-Littlewood-Sobolev and Hausdorff-Young. For standard references on these topics, see \cite{bs} or \cite{sw}. 

A related rearrangement is the following.

\begin{definition}
The signed decreasing rearrangement of $f$ is defined as 
$$
f^\circ(t)=\inf\{s\in\R:|\{x\in \Omega:f(x)>s \}|\leq{t} \},\quad t\in I_\Omega.
$$
\end{definition}
Clearly, $f^\circ$ coincides with $f^\ast$ when $f\geq{0}$ and, more generally, $|f|^\circ=f^\ast$. Further information on this rearrangement can be found in \cite{cp,ko2}.
  
Here we collect some of the basic properties of these rearrangements, the proofs for which are adapted from \cite{sw}.
\begin{lemma}
\label{lm:4}
Let $f:\Omega\rightarrow\R$ be measurable and satisfying \eqref{condition}. Then 
\begin{enumerate}
\item[a)] its signed decreasing rearrangement $f^\circ:I_\Omega\rightarrow(-\infty,\infty)$ is decreasing and equimeasurable with $f$;
\item[b)] its decreasing rearrangement $f^\ast:I_\Omega\rightarrow[0,\infty)$ is decreasing and equimeasurable with $|f|$.
\end{enumerate}
\end{lemma}

\begin{proof}
If $t_1\geq t_2$, it follows that $$\{s: |\{x\in \Omega:f(x)>s \}|\leq t_2\}\subset\{s: |\{x\in \Omega:f(x)>s \}|\leq t_1\}.$$ 
Since this is equally true for $|f|$ in place of $f$, it shows that both $f^\ast$ and $f^\circ$ are decreasing functions. 

Fix $s$. For $t\in I_\Omega$, $f^\circ(t)>s$ if and only if $t<|\{x\in \Omega:f(x)>s \}|$, from where it follows that 
$$
|\{t\in I_\Omega:f^\circ(t)>s \}|=|\{x\in \Omega:f(x)>s \}|.
$$ 
Again, applying this to $|f|$ in place of $f$ yields the corresponding statement for $f^\ast$. 

\end{proof}

One may ask how the rearrangement $f^\ast$ behaves when additional conditions are imposed on $f$. In particular, is the map $f\mapsto f^\ast$ a bounded operator on various function spaces? A well-known result in this direction is that this map is an isometry on $\Lp$, which follows immediately from Lemmas \ref{lm:3} and \ref{lm:4}. 

\begin{proposition}
\label{pr:10}
For all $1\leq{p}\leq\infty$, if $f\in \Lp(\Omega)$ then $f^\ast\in \Lp(I_\Omega)$ with $\norm{f^\ast}{\Lp(I_\Omega)}=\norm{f}{\Lp(\Omega)}$.
\end{proposition}

Another well-known result is the P\'{o}lya-Szeg\H{o} inequality, which asserts that the Sobolev norm decreases under the symmetric decreasing rearrangement (\cite{bur}), yet another kind of rearrangement. From this one can deduce the following (see, for instance, \cite{cp}).

\begin{theorem}
If $f\in W^{1,p}(\R^n)$ then
$$
n\omega_n^{1/n}\left(\int_{0}^{\infty}\!\left|\frac{d}{{d}t}f^\ast(t)\right|^pt^{p/n'}\,{d}t\right)^{\frac{1}{p}}\leq \left(\int_{\R^n}\!|\nabla{f}|^p\right)^{\frac{1}{p}},
$$
where $n'$ is the H\"{o}lder dual exponent of $n$ and $\omega_n$ denotes the volume of the unit ball in $\Rn$.
\end{theorem}

Despite these positive results, there are some closely related spaces on which the operator $f\mapsto f^\ast$ is not bounded. One such example is the John-Nirenberg space $\JNp(\Omega)$. We say that $f\in \Loneloc(\Omega)$ is in $\JNp(\Omega)$ if there exists a constant $K\geq{0}$ such that
\begin{equation}\label{jnp}
\sup\sum|Q_i|\left(\dashint_{Q_i}\!|f-f_{Q_i}| \right)^p\leq K^p,
\end{equation}
where the supremum is taken over all collections of pairwise disjoint cubes $Q_i$ in $\Omega$. We define the quantity $\norm{f}{\JNp}$ as the smallest $K$ for which (\ref{jnp}) holds. One can show that this is a norm on $\JNp(\Omega)$ modulo constants. These spaces have been considered in the case where $\Omega$ is a cube in \cite{dhky,jn} and a general Euclidean domain in \cite{hsmv}, and generalised to a metric measure space in \cite{abmy,ms}.

While it is well known that $\Lp(\Omega)\subset \JNp(\Omega)\subset L^{p,\infty}(\Omega)$, the strictness of these inclusions has only recently been addressed (\cite{abmy,dhky}). 

In the case where $\Omega=I$, a (possibly unbounded) interval, the following is obtained:
\begin{theorem}[\cite{dhky}]
Let $f:I\rightarrow\R$ be a monotone function with $f\in \Lone(I)$. Then there exists $c=c(p)>0$ such that 
$$
\norm{f}{\JNp}\geq c\norm{f-C}{\Lp}
$$
for some $C\in\R$.
\end{theorem}
In other words, monotone functions are in $\JNp(I)$ if and only if they are also in $\Lp(I)$. In \cite{dhky}, an explicit example of a function $f\in \JNp(I)\setminus \Lp(I)$ is constructed when $I$ is a finite interval. This leads to the observation that the decreasing rearrangement is not bounded on $\JNp(I)$.

\begin{corollary}
If $I$ is a finite interval, there exists an $f\in \JNp(I)$ such that $f^\ast\notin \JNp(I_I)$. 
\end{corollary}

\begin{proof}
Since $f\notin \Lp(I)$, it follows from Proposition \ref{pr:10} that $f^\ast\notin \Lp(I_I)$. As $f^\ast$ is monotone, it follows from the previous theorem that $f^\ast\notin \JNp(I_I)$. 
\end{proof}

We consider now the question of boundedness of rearrangements on $\BMO{\cS}{p}(\Omega)$ spaces. 

\begin{problem}
Does there exist a constant $c$ such that for all $f\in\BMO{\cS}{p}(\Omega)$, $\norm{f^\ast}{\BMO{}{p}(I_\Omega) }\leq c\norm{f}{\BMO{\cS}{p}(\Omega)}$? If so, what is the smallest constant, written $c_\ast(p,\cS)$, for which this holds?
\end{problem}

\begin{problem}
Does there exist a constant $c$ such that for all $f\in\BMO{\cS}{p}(\Omega)$, $\norm{f^\circ}{\BMO{}{p}(I_\Omega) }\leq c\norm{f}{\BMO{\cS}{p}(\Omega)}$? If so, what is the smallest constant, written $c_\circ(p,\cS)$, for which this holds?
\end{problem}

Clearly, if such constants exist, then they are at least equal to one. The work of Garsia-Rodemich and Bennett-DeVore-Sharpley implies an answer to the first problem and that $c_\ast(1,\cQ)\leq 2^{n+5}$ when $\Omega=\Rn$:

\begin{theorem}[\cite{bdvs,gr}]
\label{th:6}
If $f\in \BMO{}{}(\Rn)$, then $f^\ast\in\BMO{}{}\big((0,\infty)\big)$ and 
$$
\norm{f^\ast}{\BMO{}{}}\leq 2^{n+5}\norm{f}{\BMO{}{}}.
$$
\end{theorem}

These results were obtained by a variant of the Calder\'{o}n-Zygmund decomposition (\cite{st2}). Riesz' rising sun lemma, an analogous one-dimensional result that can often be used to obtain better constants, was then used by Klemes to obtain the sharp estimate that for $\Omega=I$, a finite interval, $c_{\circ}(1,\cI)=1$.

\begin{theorem}[\cite{kl}]
If $f\in \BMO{}{}(I)$, then $f^\circ\in\BMO{}{}(I_I)$ and 
$$
\norm{f^\circ}{\BMO{}{}}\leq \norm{f}{\BMO{}{}}.
$$
\end{theorem}

An elementary but key element of Klemes' proof that can be generalised to our context of general shapes is the following shapewise identity. 

\begin{lemma}
\label{mean}
For any shape $S$, if $f\in \Lone(S)$ then 
$$
\dashint_{S}\!|f-f_S|=\frac{2}{|S|}\int_{\{f>f_S \}}\!(f-f_S)=\frac{2}{|S|}\int_{\{f<f_S \}}\!(f_S-f).
$$
\end{lemma}

\begin{proof} Write
$$
\int_{S}\!|f(x)-f_S|\,dx=\int_{\{x\in{S}: f(x)>f_S \}}\!(f(x)-f_S)\,dx+\int_{\{x\in{S}:f(x)<f_S \}}\!(f_S-f(x))\,dx.
$$
Since
$$
\int_{\{x\in{S}: f(x)>f_S\}}\!(f(x)-f_S)\,dx+ \int_{\{x\in{S}:f(x)<f_S\}}\!(f(x)-f_S)\,dx=\int_{S}\!(f(x)-f_S)\,dx=0,
$$
it follows that
$$
\int_{\{x\in{S}:f(x)>f_S\}}\!(f(x)-f_S)\,dx= \int_{\{x\in{S}:f(x)<f_S\}}\!(f_S-f(x))\,dx,
$$
which gives the identity.
\end{proof}

The next sharp result concerning rearrangements is due to Korenovskii, showing that for $\Omega=I$, a finite interval, $c_{\ast}(1,\cI)=1$. The proof of this result makes direct use of Klemes' theorem.

\begin{theorem}[\cite{ko1}]\label{th:3}
If $f\in \BMO{}{}(I)$, then $f^\ast\in\BMO{}{}(I_I)$ and 
$$
\norm{f^\ast}{\BMO{}{}}\leq \norm{f}{\BMO{}{}}.
$$
\end{theorem}

Important in Korenovskii's transition from a sharp estimate for $c_\circ(1,\cS)$ to one for $c_\ast(1,\cS)$ is the fact that $|f|^\circ=f^\ast$, bringing us to consider the boundedness of the absolute value operator. Recall from Example~\ref{ex-absval} that $F(x)=|x|$ gives us the sharp shapewise inequality in Proposition \ref{pr:2} with $p=1$, which implies that $\norm{|f| }{\BMO{\cS}{}}\leq 2\norm{f}{\BMO{\cS}{}}$. However, this need not be sharp as a norm inequality, and so it is natural to ask

\begin{problem}
What is the smallest constant $c_{|\cdot|}(p,\cS)$ such that $\norm{|f|}{\BMO{\cS}{p}}\leq c_{|\cdot|}(p,\cS)\norm{f}{\BMO{\cS}{p}}$ holds for all $f\in \BMO{\cS}{p}(\Omega)$?
\end{problem}

It is clear that $c_{|\cdot|}(p,\cS)\geq{1}$ and Proposition \ref{pr:2} implies that $c_{|\cdot|}(p,\cS)\leq 2$. Applying this estimate along with Klemes' theorem yields the non-sharp bound $c_\ast(1,\cI)\leq{2}$. 

In order for Korenovskii to obtain a sharp result for $c_\ast(1,\cI)$, a more subtle argument was needed that allowed him to conclude that $c_{|\cdot|}(1,\cI)={1}$ when $\Omega=I$:
 
\begin{theorem}[\cite{ko1}]
\label{thm-absvalI}
If $f\in \BMO{}{}(I)$, then $\norm{|f|}{\BMO{}{}}\leq \norm{f}{\BMO{}{}}.$
\end{theorem} 

The following is one of the essential parts of this argument. It demonstrates that the behaviour of the absolute value operator is more easily analyzed for decreasing functions. 

\begin{theorem}[\cite{ko1}]
\label{th:1}
Let $I$ be a finite interval and $f\in \Lone(I)$ be a decreasing function. Then, 
$$
\dashint_{I}\!\big||f|-|f|_{I} \big|\leq \sup_{J\subset I}\dashint_{J}\!|f-f_{J}|
$$
where the supremum is taken over all subintervals $J$ of $I$. 
\end{theorem}

Further sharp results were obtained by Korenovskii in the case where $\Omega=R$, a rectangle, and $\cS=\cR$: it was shown that, similar to the one-dimensional case just discussed, 
$c_{|\cdot|}(1,\cR)=c_{\circ}(1,\cR)=c_{\ast}(1,\cR)=1$. 

\begin{theorem}[\cite{ko3}]
\label{thm-absvalR}
If $f\in \BMO{\cR}{}(R)$, then $|f|\in\BMO{\cR}{}(R)$ and $f^\circ,f^\ast\in\BMO{}{}(I_R)$ with 
$$
\norm{|f|}{\BMO{\cR}{}}\leq \norm{f}{\BMO{\cR}{}},\qquad\norm{f^\circ}{\BMO{}{}}\leq \norm{f}{\BMO{\cR}{}},\qquad\norm{f^\ast}{\BMO{}{}}\leq \norm{f}{\BMO{\cR}{}}.
$$
\end{theorem}

This demonstrates the paradigm that rectangles behave more similarly to one-dimensional intervals than cubes do. In particular, the generalization of Klemes' theorem to the higher-dimensional case of rectangles (the result that $c_{\circ}(1,\cR)=1$) employs a multidimensional analogue of Riesz' rising sun lemma using rectangles (\cite{kls}) when such a theorem could not exist for arbitrary cubes.

Following the techniques of \cite{ko1}, general relationships can be found between the constants $c_{|\cdot|}(1,\cS)$, $c_{\circ}(1,\cS)$, $c_{\ast}(1,\cS)$ for an arbitrary basis of shapes. First, we show that $c_{|\cdot|}(1,\cS)\leq c_{\circ}(1,\cS)$.

\begin{proposition}
\label{pr:7}
For any collection of shapes $\cS$, if $f\in \BMO{\cS}{}(\Omega)$ and $\norm{f^\circ}{\BMO{}{}(I_\Omega) }\leq c\norm{f}{\BMO{\cS}{}(\Omega)}$, then $\norm{|f|}{\BMO{\cS}{}(\Omega)}\leq c\norm{f}{\BMO{\cS}{}(\Omega)}$ for the same constant $c$.
\end{proposition}

\begin{proof}
Fix a shape $S\in\cS$ and assume that $f\in \BMO{\cS}{}(\Omega)$ is supported on $S$. 

Since $f$ is equimeasurable with ${f}^\circ$ by Lemma \ref{lm:4}, it follows from Lemma \ref{lm:1} that $|f|$ is equimeasurable with $|f^\circ|$ (recall that $|S|<\infty$). Writing ${E}=(0,|S|)$, by Lemma \ref{lm:3} we have that $|f^\circ|_{E}=|f|_{S}$ and also, then, that
$$
\int_{0}^{|S|}\!\big||f^\circ|-|f^\circ|_{{E}}\big| =\int_{0}^{|S|}\!\big||f^\circ|-|f|_{{S}}\big|=\int_{S}\!\big||f|-|f|_{S}\big|.
$$ 

Thus, by Theorem \ref{th:1},
$$
\dashint_{S}\!\big||f|-|f|_S\big|=\dashint_{(0,|S|)}\!\big||{f}^\circ|-|{f}^\circ|_{E}\big|\leq\sup_{J\subset{(0,|S|) }}\dashint_{J}\!|{f}^\circ-({f}^\circ)_{J}|.
$$
For all $J\subset{(0,|S|) }$, we have that
$$
\dashint_{J}\!|f^\circ-(f^\circ)_{J}|\leq \norm{f^\circ}{\BMO{}{}\big((0,|S|)\big) }\leq \norm{f^\circ}{\BMO{}{}(I_\Omega) }\leq c\norm{f}{\BMO{\cS}{}(\Omega)},
$$
and, therefore,
$$
\dashint_{S}\!\big||f|-|f|_S\big|\leq c\norm{f}{\BMO{\cS}{}(\Omega)}.
$$
Taking a supremum over all shapes $S\in\cS$ yields the result. 

\end{proof}

This result in turn allows us to prove the following relationship: 
$$
c_\ast(1,\cS)\leq c_{|\cdot|}(1,\cS)c_\circ(1,\cS)\leq c_\circ(1,\cS)^2.
$$

\begin{proposition}
\label{pr:16}
For any collection of shapes $\cS$, if $f\in \BMO{\cS}{}(\Omega)$ and $\norm{f^\circ}{\BMO{}{}(I_\Omega) }\leq c\norm{f}{\BMO{\cS}{}(\Omega)}$, then $\norm{f^\ast}{\BMO{}{}(I_\Omega) }\leq c^2\norm{f}{\BMO{\cS}{}(\Omega)}$. 
\end{proposition}

\begin{proof}
By Proposition \ref{pr:7}, it follows that 
$$
\norm{|f|}{\BMO{\cS}{}(\Omega) }\leq c \norm{f}{\BMO{\cS}{}(\Omega) }.
$$
Writing $f^\ast=|f|^\circ$, we have that 
$$
\norm{f^\ast}{\BMO{}{}(I_\Omega)} = \norm{|f|^\circ}{\BMO{}{}(I_\Omega)}\leq c\norm{|f|}{\BMO{\cS}{}(\Omega) }\leq c^2 \norm{f}{\BMO{\cS}{}(\Omega)}.
$$
\end{proof}

From these results, we see that a sharp result of the form $c_\circ(1,\cS)=1$ would immediately imply two more sharp results, $c_{|\cdot|}(1,\cS)=1$ and $c_\ast(1,\cS)=1$.

Although the dyadic cubes do not, in general, cover a domain, the space dyadic $\BMO{}{}$ has been extensively studied in the literature (see \cite{gar} for an early work illustrating its connection to martingales). In fact, many of the results in this section hold for that space; as such, extending our notation to include $\cS=\cQ_d$ even though it does not form a basis, we provide here a sample of the known sharp results.

Klemes' theorem was extended to the higher-dimensional dyadic case by Nikolidakis, who shows, for $\Omega=Q$, that $c_{\circ}(1,\cQ_d)\leq 2^n$:

\begin{theorem}[\cite{ni}]
\label{th:5}
If $f\in \BMO{\cQ_d}{}(Q)$, then $f^\circ\in\BMO{}{}(I_Q)$ and 
$$
\norm{f^\circ}{\BMO{}{}}\leq 2^n\norm{f}{\BMO{\cQ_d}{}}.
$$
\end{theorem}

As a corollary of Proposition \ref{pr:16} and Theorem \ref{th:5} we have the following, which shows that $c_{\ast}(1,\cQ_d)\leq 2^{n+1}$, an improvement on Theorem \ref{th:6}.

\begin{corollary}
If $f\in \BMO{\cQ_d}{}(Q)$, then $f^\ast\in\BMO{}{}(I_Q)$ and 
$$
\norm{f^\ast}{\BMO{}{}}\leq 2^{n+1}\norm{f}{\BMO{\cQ_d}{}}.
$$
\end{corollary}

The previous discussion emphasized the situation when $p=1$. For $p=2$, even more powerful tools are available: using probabilistic methods, Stolyarov, Vasyunin and Zatitskiy prove the following sharp result.

\begin{theorem}[\cite{svz}]
$c_\ast(2,\cQ_d)=\frac{1+2^n}{2^{1+n/2}}$.
\end{theorem}
\section{{\bf Truncations}}

An immediate consequence of the bounds for the absolute value is the following result demonstrating that $\BMO{}{}$ is a lattice.

\begin{proposition}
\label{lm:2}
For any basis of shapes, if $f_1,f_2$ are real-valued functions in $\BMO{\cS}{p}(\Omega)$, then $\max(f_1,f_2)$ and $\min(f_1,f_2)$ are in  $\BMO{\cS}{p}(\Omega)$, with
$$
\norm{ \max(f_1,f_2) }{\BMO{\cS}{p}}\leq \frac{1+c_{|\cdot|}(p,\cS)}{2}\left( \norm{f_1}{\BMO{\cS}{p}}+\norm{f_2}{\BMO{\cS}{p}}\right).
$$
and
$$
\norm{ \min(f_1,f_2) }{\BMO{\cS}{p}}\leq \frac{1+c_{|\cdot|}(p,\cS)}{2}\left( \norm{f_1}{\BMO{\cS}{p}}+\norm{f_2}{\BMO{\cS}{p}}\right).
$$
\end{proposition}

\begin{proof}
This follows from writing 
$$
\max(f_1,f_2)=\frac{(f_1+f_2)+ |f_1-f_2| }{2}\;\;\;\text{and}\;\;\; \min(f_1,f_2)=\frac{(f_1+f_2)- |f_1-f_2| }{2}
$$
and using the estimate for the absolute value:
$$
\norm{|f_1-f_2|}{\BMO{\cS}{p}} \leq c_{|\cdot|}(p,\cS)(\norm{f_1}{\BMO{\cS}{p}}+\norm{f_2}{\BMO{\cS}{p}}).
$$
\end{proof}

In particular, applying Theorems~\ref{thm-absvalI} and \ref{thm-absvalR}, this yields the sharp constant $1$ for $p=1$ when $\cS=\cI$ or $\cR$. 

We can also obtain the sharp constant $1$ via a sharp shapewise inequality for the cases $p=1$ and $p=2$, regardless of the basis.  The proof of the case $p = 2$ in the following result is given by Reimann and Rychener \cite{rr} for the basis $\cS = \cQ$.

\begin{proposition}
\label{pr:17}
Let $p = 1$ or $p = 2$.  
Let $\cS$ be any basis of shapes and fix a shape $S\in\cS$.  If $f_1,f_2\in\BMO{\cS}{p}(\Omega)$ and $f=\max(f_1,f_2)$ or $f=\min(f_1,f_2)$, we have
\begin{equation}
\label{eqn-max}
\dashint_{S}\!|f-f_{S}|^p
\leq \dashint_{S}|f_1-(f_1)_{S}|^p+\dashint_{S}|f_2-(f_2)_{S}|^p.
\end{equation} 
Consequently
$$
\norm{ \max(f_1,f_2) }{\BMO{\cS}{p}}\leq \norm{f_1}{\BMO{\cS}{p}}+\norm{f_2}{\BMO{\cS}{p}}
$$
and
$$
\norm{ \min(f_1,f_2) }{\BMO{\cS}{p}}\leq \norm{f_1}{\BMO{\cS}{p}}+\norm{f_2}{\BMO{\cS}{p}}.
$$
\end{proposition}

\begin{proof}
First, for $p = 1$, fix a shape $S\in\cS$ and let $f=\min(f_1,f_2)$.  Note that $f_S \leq (f_i)_S$ as $f\leq f_i$ on $S$  for $i = 1,2$. 
Let 
$$E_1 = \{x \in S: f_1(x) \leq f_2(x)\} =  \{x \in S: f(x) = f_1(x)\}, \quad E_2 = S \setminus E_1.$$
By Lemma \ref{mean},
\begin{eqnarray*}
\dashint_{S}|f(x)-f_S|\,dx
&=& \frac{2}{|S|}\int_{\{x\in{S}:f(x)<f_S \} }(f_S-f(x))\,dx\\
&=& \frac{2}{|S|} \sum_{i = 1}^2\int_{\{x\in{E_i}:f_i(x)<f_S \} }(f_S-f_i(x))\,dx\\
&\leq &\frac{2}{|S|} \sum_{i = 1}^2\int_{\{x \in S:f_i(x)<(f_i)_S \} }((f_i)_S-f_i(x))\,dx\\
&=&\dashint_{S}|f_1-(f_1)_{S}|+\dashint_{S}|f_2-(f_2)_{S}|.
\end{eqnarray*}
For $f=\max(f_1,f_2)$, the previous arguments follow in a similar way, except that we apply Lemma \ref{mean} to write the mean oscillation in terms of an integral over the set $\{x\in{S}:f(x)>f_S \}$.

For $p = 2$, we include, for the benefit of the reader, the proof from \cite{rr}, with cubes replaced by shapes.  Let $f=\max(f_1,f_2)$.  We may assume without loss of generality that $(f_1)_S  \geq (f_2)_S$.  Consider the sets 
$$S_1=\{x\in{S}:f_{2}(x)<f_1(x) \}, \quad S_2=\{x\in S: f_2(x) \geq f_1(x) \mbox{ and } f_2(x)\geq(f_1)_{S}\}$$
and
$$S_3=S\setminus(S_1\cup S_2) = \{x\in{S}:f_1(x)  \leq f_{2}(x)<(f_1)_{S}\}.$$
 Then
\[
\begin{split}
\int_{S}\!|f-(f_1)_{S}|^2
&=\int_{S_1}\!|f_1(x)-(f_1)_{S}|^2\,dx+\int_{S_2}\!|f_2(x)-(f_1)_{S}|^2\,dx + \int_{S_3}\!|f_2(x)-(f_1)_{S}|^2\,dx\\
&\leq\int_{S_1}\!|f_1(x)-(f_1)_{S}|^2\,dx+\int_{S_2}\!|f_2(x)-(f_2)_{S}|^2\,dx + \int_{S_3}\!|f_1(x)-(f_1)_{S}|^2\,dx\\
&\leq\int_{S}\!|f_1(x)-(f_1)_{S}|^2\,dx+\int_{S}\!|f_2(x)-(f_2)_{S}|^2\,dx\\
\end{split}
\]
Using Proposition~\ref{pr:12} and dividing by $|S|$ gives \eqref{eqn-max}. Similarly, the result can be shown for $\min(f_1,f_2)$.
\end{proof}

For a real-valued measurable function $f$ on $\Omega$, define its truncation from above at level $k \in \R$ by
$$f^{k} = \min(f,k)$$
and its truncation from below at level $j \in \R$ as
$$
f_{j}=\max(f,j).
$$

We use the preceding propositions to prove boundedness of the upper and lower truncations on $\BMO{\cS}{p}(\Omega)$.

\begin{proposition}
\label{pr:15}
Let $\cS$ be any basis of shapes and fix a shape $S\in\cS$. 
If $f\in\BMO{\cS}{p}(\Omega)$ then for all $k, j \in \R$, 
\begin{equation}
\label{eqn-shapewiseTr}
\max\left(\dashint_{S}|f^k(x)-(f^k)_S|\,dx,  \dashint_{S}|f_j(x)-(f_j)_S|\,dx \right)\leq c\dashint_{S}|f(x)-f_S|\,dx
\end{equation}
where
$$
c=\begin{cases}
1,&p=1\,\text{or}\,\,p=2\\
\min\left(2, \frac{1+c_{|\cdot|}(p,\cS)}{2}\right),&\text{otherwise}.
\end{cases}
$$
Consequently
$$\norm{f^{k}}{\BMO{\cS}{p}}\leq c\norm{f}{\BMO{\cS}{p}} \mbox{ and } \norm{f_{j}}{\BMO{\cS}{p}}\leq c\norm{f}{\BMO{\cS}{p}}.$$ 
\end{proposition}

\begin{proof}
For the truncation from above, observing that $|f^{k}(x)-f^{k}(y)|\leq |f(x)-f(y)|$ almost everywhere and for any $k$, Proposition \ref{pr:9} gives \eqref{eqn-shapewiseTr} with $c = 2$. On the other hand, applying Proposition \ref{lm:2} to  $f^{k}=\min(f,k)$ and using the fact that constant functions have zero mean oscillation, we get the constant
$c =  \frac{1+c_{|\cdot|}(p,\cS)}{2}$.
These calculations hold for any $p\geq{1}$. 

For $p=1$ and $p = 2$, we are able to strengthen this by deriving the sharp shapewise inequality, namely \eqref{eqn-shapewiseTr} with $c = 1$, from \eqref{eqn-max} in Proposition \ref{pr:17} and the fact that constants have zero mean oscillation.
Alternatively, for $p=2$, fix a shape $S\in\cS$ and assume, without loss of generality, that $f_S=0$. Then, by Proposition \ref{pr:12},
$$
\dashint_{S}\!|f^{(k)}-(f^{(k)})_S|^2\leq\dashint_{S}\!|f^{(k)}-f_S|^2=\dashint_{S}\!|f^{(k)}|^2\leq\dashint_{S}\!|f|^2=\dashint_{S}\!|f-f_S|^2.
$$
The calculations for the truncation from below are analogous or can be derived by writing $f_j = -\min(-f,-j)$.
\end{proof}

For $p=1$ and $p=2$, this is clearly a sharp result: take any bounded function and either $k>\sup f$ or $j< \inf f$.

Also note that for $p = 1$, the sharp shapewise inequalities \eqref{eqn-max} and \eqref{eqn-shapewiseTr} give the sharp shapewise inequality \eqref{eqn-Holderp} (with constant $2$)
for the absolute value by writing
$$|f| = f_+ + f_- = \max(f_+, f_-), \quad f_+ = f_0, \; f_- = -(f^0).$$

For a measurable function $f$ on $\Omega$, define its (full) truncation at level $k$ as
$$
\Tr{(f,k)}(x) =\begin{cases}
k,& f(x)>k\\
f(x),&-k\leq f(x)\leq k\\
-k, &f(x)<-k.
\end{cases}
$$
Note that $\Tr{(f,k)}=(f^k)_{-k}=(f_{-k})^{k}$ and that $\Tr{(f,k)}\in \Linfty(\Omega)$ for each $k$. Moreover, $\Tr{(f,k)}\rightarrow{f}$ pointwise and in $\Loneloc(\Omega)$ as $k\rightarrow\infty$ (\cite{jo}). 

For a function $f\in \BMO{}{}(\Omega)$, it is not true that $\Tr{(f,k)}\rightarrow{f}$ in $\BMO{}{}(\Omega)$, unless $f\in\VMO{}{}(\Omega)$, the space of functions of vanishing mean oscillation (\cite{bn}). 
Nonetheless, as shown in Corollary \ref{co:2} below,  $\norm{\Tr{(f,k)}}{\BMO{}{}}\rightarrow\norm{f}{\BMO{}{}}$ (as mentioned without proof in \cite{st,ss}), and in fact for any basis $\cS$.

\begin{problem}
Does there exist a constant $c$ such that, for all $k$ and for all $f\in \BMO{\cS}{p}(\Omega)$, $\norm{\Tr{(f,k)}}{\BMO{\cS}{p}}\leq c\norm{f}{\BMO{\cS}{p}}$, and if so, what is the smallest constant $c_{T}(p,\cS)$ for which this holds?
\end{problem}

If $c_{T}(p,\cS)$ exists, then clearly $c_{T}(p,\cS)\geq{1}$. The results above provide a positive answer to the first question and some information about $c_{T}(p,\cS)$.

\begin{corollary}
\label{co:2}
For any choice of shapes, if $f\in\BMO{\cS}{p}(\Omega)$ then $\Tr{(f,k)}\in \BMO{\cS}{p}(\Omega)$ for all $k$ and $\norm{\Tr{(f,k)}}{\BMO{\cS}{p}}\leq c\norm{f}{\BMO{\cS}{p}}$, where
$$
c=\begin{cases}
1,&p=1\,\text{or}\,\,p=2\\
\min\left(2, \left(\frac{1+c_{|\cdot|}(p,\cS)}{2}\right)^2\right),&\text{otherwise}
\end{cases}.
$$ 
Moreover, for $p = 1$ or $p = 2$, $\norm{\Tr{(f,k)}}{\BMO{\cS}{p}} \ra \norm{f}{\BMO{\cS}{p}}$ as $k \ra \infty$.
\end{corollary}

\begin{proof} Writing $\Tr{(f,k)}=(f^k)_{-k}$ and applying inequality \eqref{eqn-shapewiseTr} in Proposition \ref{pr:15} gives us the shapewise inequality
\begin{equation}
\label{eqn-twosidedTr}
\dashint_{S}\!|\Tr{(f,k)}-(\Tr{(f,k)})_S|^p\leq c\dashint_{S}\!|f-f_S|^p
\end{equation}
with constant $c = 1$ for $p=1$ and $p=2$, and $c \leq \left(\min\Big(2,\frac{1+c_{|\cdot|}(p,\cS)}{2}\Big)\right)^2$.   We
get $c \leq 2$ (as opposed to $2^2$) by applying  Proposition~\ref{pr:9}  directly with the estimate $|\Tr{(f,k)}(x)-\Tr{(f,k)}(y)|\leq|f(x)-f(y)|$. 

For the convergence of the norms in the case $p=1$ and $p=2$, we have that $|\Tr{(f,k)}| \leq |f| \in \Lone(S)$ implies $(\Tr{(f,k)})_S \ra f_S$, and since $\Tr{(f,k)} \ra f$
pointwise a.e.\ on $S$, we can apply Fatou's lemma and \eqref{eqn-twosidedTr} with $c = 1$ to get
\[
\begin{split}
\dashint_{S}\!|f-f_S|^p
&\leq \liminf_{k \ra \infty} \dashint_{S}\!|\Tr{(f,k)}-(\Tr{(f,k)})_S|^p\\
&\leq\limsup_{k \ra \infty} \dashint_{S}\!|\Tr{(f,k)}-(\Tr{(f,k)})_S|^p \leq \dashint_{S}\!|f-f_S|^p.\\
\end{split}
\]
\end{proof}

This result gives the sharp constant for $p=1,2$ and an upper bound for $c_{T}(p,\cS)$. Of course, the known upper bound $c_{|\cdot|}(p,\cS)\leq 2$ implies 
$$
\left(\frac{1+c_{|\cdot|}(p,\cS)}{2}\right)^2\leq \frac{9}{4}
$$
(which appears, for example, in Exercise 3.1.4 in \cite{gra}), but this is worse than the truncation bound $c_{T}(p,\cS)\leq 2$.  On the other hand, if $c_{|\cdot|}(p,\cS)\leq 2\sqrt{2}-1$, then the bound depends on $c_{|\cdot|}(p,\cS)$. In particular, a result of $c_{|\cdot|}(p,\cS)=1$ would imply that $c_{T}(p,\cS)=1$. 
 
\section{{\bf The John-Nirenberg inequality}}

We now come to the most important inequality in the theory of $\BMO{}{}$, originating in the paper of John and Nirenberg \cite{jn}.

\begin{definition}  Let $X \subset \Rn$ be a set of finite Lebesgue measure.  We say that $f \in L^1(X)$ satisfies the John-Nirenberg inequality on $X$ if  there exist constants 
$c_1,c_2>0$ such that
\begin{equation}
\label{decay}
|\{x\in X:|f(x)-f_{X}|>\alpha \}|\leq c_1|X|e^{-c_2\alpha},\qquad \alpha>0.
\end{equation}
\end{definition}

The following is sometimes referred to as the John-Nirenberg Lemma.

\begin{theorem}[\cite{jn}]
\label{thm-jn}
If $X = Q$, a cube in $\Rn$, then there exist constants $c$ and $C$ such that for all $f \in \BMO{}{}(Q)$, \eqref{decay} holds with $c_1 = c$, $c_2 = C/\norm{f}{\BMO{}{}}$.
\end{theorem}

More generally, given a basis of shapes $\cS$ on a domain $\Omega \subset \Rn$, $|\Omega|<\infty$, one can pose the following problem.

\begin{problem}
Does $f\in \BMO{\cS}{p}(\Omega)$ imply that $f$ satisfies the John-Nirenberg inequality on $\Omega$?  That is, do there exist constants $c,C>0$ such that 
$$
|\{x\in \Omega:|f(x)-f_{\Omega}|>\alpha \}|\leq c|\Omega|\exp\left(-\frac{C}{\norm{f}{\BMO{\cS}{p}}}\alpha \right),\qquad \alpha>0
$$
holds for all $f\in \BMO{\cS}{p}(\Omega)$?
If so, what is the smallest constant $c_{\Omega,\scriptscriptstyle{\rm{JN}}}(p,\cS)$ and the largest constant $C_{\Omega,\scriptscriptstyle{\rm{JN}}}(p,\cS)$ for which this inequality holds? 
\end{problem}

When $n=1$, $\Omega=I$, a finite interval, and $\cS = \cI$, the positive answer is a special case of Theorem~\ref{thm-jn}. Sharp constants are known for the cases $p = 1$
and $p = 2$. For $p = 1$, it is shown in \cite{le2} that $c_{I,\scriptscriptstyle{\rm{JN}}}(1,\cI)=\frac{1}{2}e^{4/e}$ and in \cite{ko1} that $C_{I,\scriptscriptstyle{\rm{JN}}}(1,\cI)=2/e$. For $p = 2$, Bellman function techniques are used in \cite{vv} to give $c_{I,\scriptscriptstyle{\rm{JN}}}(2,\cI)=4/e^2$ and $C_{I,\scriptscriptstyle{\rm{JN}}}(2,\cI)=1$.

When $n\geq{2}$, $\Omega=R$, a rectangle, and $\cS=\cR$, a positive answer is provided by a less well-known result due to Korenovskii in \cite{ko3}, where he also shows the sharp constant $C_{R,\scriptscriptstyle{\rm{JN}}}(1,\mathcal{R})=2/e$.

Dimension-free bounds on these constants are also of interest. In \cite{css1}, Cwikel, Sagher, and Shvartsman conjecture a geometric condition on cubes and prove dimension-free bounds for $c_{\Omega,\scriptscriptstyle{\rm{JN}}}(1,\cQ)$ and $C_{\Omega,\scriptscriptstyle{\rm{JN}}}(1,\cQ)$ conditional on this hypothesis being true.

Rather than just looking at $\Omega$, we can also consider whether the John-Nirenberg inequality holds for all shapes $S$.

\begin{definition}
We say that a function $f \in \Loneloc(\Omega)$ has the John-Nirenberg property with respect to a basis  $\cS$ of shapes on $\Omega$ if 
there exist constants $c_1,c_2>0$ such that for all $S\in\cS$, \eqref{decay} holds for $X = S$.
\end{definition}

We can now formulate a modified problem.

\begin{problem}
For which bases $\cS$ and $p \in [1,\infty)$ does $f\in \BMO{\cS}{p}(\Omega)$ imply that $f$ has the John-Nirenberg property with respect to $\cS$?  If this is the case, what is the smallest constant $c = c_{\scriptscriptstyle{\rm{JN}}}(p,\cS)$ and the largest constant $C = C_{\scriptscriptstyle{\rm{JN}}}(p,\cS)$ for which \eqref{decay} holds for all $f\in \BMO{\cS}{}(\Omega)$ and $S \in \cS$ with $c_1 = c$, $c_2 = C/\norm{f}{\BMO{\cS}{p}}$?
\end{problem}

Since Theorem~\ref{thm-jn} holds for any cube $Q$ in $\Rn$ with constants independent of $Q$,  it follows that for a domain $\Omega \subset \Rn$, any $f \in \BMO{}{}(\Omega)$ has the John-Nirenberg property with
respect to $\cQ$, and equivalently $\cB$. Similarly, every $f \in \BMO{\cR}{}(\Omega)$ has the John-Nirenberg property with respect to $\cR$.

In the negative direction, $f\in\BMO{\cC}{p}(\Rn)$ does not necessarily have the John-Nirenberg property with respect to $\cC$ (\cite{la,ly}). This space, known in the literature as $\CMO{}{}$ for central mean oscillation or $\CBMO{}{}$ for central bounded mean oscillation, was originally defined with the additional constraint that the balls have radius at least 1 (\cite{cl,gc}).

We now state the converse to Theorem~\ref{thm-jn}, namely that the John-Nirenberg property is sufficient for $\BMO{}{}$, in more generality.

\begin{theorem}
 If $f \in \Loneloc(\Omega)$ and $f$ has the John-Nirenberg property with respect to $\cS$,
 then $f\in \BMO{\cS}{p}(\Omega)$  for all $p \in [1, \infty)$, with 
 $$\norm{f}{\BMO{\cS}{p}} \leq \frac{(c_1 p\Gamma(p))^{1/p}}{c_2}.$$
\end{theorem}

\begin{proof}
Take $S\in\cS$. By Cavalieri's principle and \eqref{decay},
\[
\begin{split}
\dashint_{S}\!|f-f_S|^{p}
&=\frac{p}{|S|}\int_{0}^{\infty}\!\alpha^{{p}-1}|\{x\in S:|f(x)-f_S|>\alpha \}|\,d\alpha\\
&\leq pc_1 \int_{0}^{\infty}\!\alpha^{{p}-1}\exp\left(-c_2\alpha \right) \,d\alpha\\ 
&= \frac{c_1 p\Gamma(p)}{{c_2}^{p}},
\end{split}
\]
from where the result follows.
\end{proof}

By Lemma~\ref{pr:18}, this theorem shows that every $f$ with the John-Nirenberg property is in $\Lploc(\Omega)$.

A consequence of the John-Nirenberg Lemma, Theorem~\ref{thm-jn}, is that $$\BMO{}{p_1}(\Rn)\cong \BMO{}{p_2}(\Rn)$$ for all $1\leq p_1,p_2<\infty$.  This can be stated in more generality
as a corollary of the preceding theorem.

\begin{corollary}
\label{cor-jn}
If there exists $p_0\in [1,\infty)$ such that every $f\in \BMO{\cS }{p_0}(\Omega)$  has the John-Nirenberg property with respect to $\cS$, then
$$\BMO{\cS}{p_1}(\Omega)\cong \BMO{\cS}{p_2}(\Omega), \quad p_0 \leq p_1,p_2<\infty.$$
\end{corollary}

\begin{proof}
The hypothesis means that there are constants $c_{\scriptscriptstyle{\rm{JN}}}(p_0,\cS)$, $C_{\scriptscriptstyle{\rm{JN}}}(p_0,\cS)$ such that if $f\in \BMO{\cS }{p_0}(\Omega)$ then $f$ satisfies \eqref{decay}  for all $S \in \cS$, with $c_1 = c_{\scriptscriptstyle{\rm{JN}}}(p_0,\cS)$, $c_2 = C_{\scriptscriptstyle{\rm{JN}}}(p_0,\cS)/\norm{f}{\BMO{\cS}{p_0}}$.

From the  preceding theorem, this implies that $\BMO{\cS }{p_0}(\Omega) \subset \BMO{\cS}{p}(\Omega)$  for all $p \in [1, \infty)$, with
$$\norm{f}{\BMO{\cS}{p}} \leq \frac{(c_{\scriptscriptstyle{\rm{JN}}}(p_0,\cS) p\Gamma(p))^{1/p}}{C_{\scriptscriptstyle{\rm{JN}}}(p_0,\cS)} \norm{f}{\BMO{\cS}{p_0}}.
$$

Conversely, Proposition \ref{pr:1} gives us that $\BMO{\cS}{p}(\Omega)\subset \BMO{\cS}{p_0}(\Omega)$ whenever $p_0 \leq p <\infty$, with $\norm{f}{\BMO{\cS}{p_0}}\leq \norm{f}{\BMO{\cS}{p}}$.

Thus all the spaces $\BMO{\cS}{p}(\Omega)$, $p_0 \leq p <\infty$, are congruent to $\BMO{\cS}{p_0}(\Omega)$.
\end{proof}

The John-Nirenberg Lemma gives the hypothesis of Corollary~\ref{cor-jn} for the bases $\cQ$ and $\cB$ on $\Rn$ with $p_0 = 1$.  As pointed out, by results of \cite{ko3} this also applies to the basis $\cR$, showing that $\BMO{\cR}{p_1}(\Omega)\cong \BMO{\cR}{p_2}(\Omega)$ for all $1\leq p_1,p_2<\infty$.

\begin{problem}
If the hypothesis of Corollary~\ref{cor-jn} is satisfied with $p_0 = 1$, what is the smallest constant $c_{}(p,\cS)$ such that $\norm{f}{\BMO{\cS}{p}}\leq c_{}(p,\cS)\norm{f}{\BMO{\cS}{}}$ holds for all $f\in \BMO{\cS}{p}(\Omega)$?
\end{problem}

The proof of Corollary~\ref{cor-jn}  gives a well-known upper bound on $c_{}(p,\cS)$:
$$c_{}(p,\cS) \leq \frac{\big(c_{\scriptscriptstyle{\rm{JN}}}(1,\cS)p\Gamma(p)\big)^{1/{p}}}{C_{\scriptscriptstyle{\rm{JN}}}(1,\cS)}.$$  

\section{{\bf Product decomposition}}
In this section, we assume that the domain $\Omega$ can be decomposed as 
\begin{equation}
\label{eqn-decomp}
\Omega=\Omega_1\times \Omega_2\times\cdots\times \Omega_k
\end{equation}
 for $2\leq k\leq{n}$, where each $\Omega_i$ is a domain in $\R^{m_i}$ for $1\leq {m_i}\leq{n-1}$, having its own basis of shapes $\cS_i$.

We will require some compatibility between the basis $\cS$ on all of $\Omega$ and these individual bases. 
\begin{definition}
We say that $\cS$ satisfies the weak decomposition property with respect to $\{\cS_i\}_{i=1}^{k}$ if for every $S\in\cS$ there exist $S_i\in\cS_i$ for each $i=1,\ldots,k$ such that $S=S_1\times S_2\times\ldots S_{k}$. If, in addition, for every $\{S_i \}_{i=1}^{k}$, $S_i\in\cS_i$, the set $S_1\times S_2\times\ldots S_{k}\in\cS$, then we say that the basis $\cS$ satisfies the strong decomposition property with respect to $\{\cS_i\}_{i=1}^{k}$.
\end{definition}
Using $\cR_i$ to denote the basis of rectangles in $\Omega_i$ (interpreted as $\cI_{i}$ if $m_i=1$), note that the basis $\cR$ on $\Omega$ satisfies the strong decomposition property with respect to $\{\cR_i\}_{i=1}^{k}$ and for any $k$. Meanwhile, the basis $\cQ$ on $\Omega$ satisfies the weak decomposition property with respect to $\{\cR_i\}_{i=1}^{k}$ for any $k$, but not the strong decomposition property.

We now turn to the study of the spaces $\BMO{\cS_{i}}{p}{(\Omega)}$, first defined using different notation in the context of the bidisc $\T\times \T$ in \cite{cs}. For simplicity, we only define $\BMO{\cS_{1}}{p}{(\Omega)}$, as the other $\BMO{\cS_{i}}{p}{(\Omega)}$ for $i=2,\ldots,k$ are defined analogously. We write a point in $\Omega$ as $(x,y)$, where $x\in \Omega_1$ and $y\in \Omegatil= \Omega_2\times\cdots\times \Omega_k$. Writing $f_y(x)=f(x,y)$, functions in $\BMO{\cS_{1}}{p}{(\Omega)}$ are those for which $f_y$ is in $\BMO{\cS_{1}}{p}{(\Omega_1)}$, uniformly in $y$:

\begin{definition}
\label{def-cS1}
A function $f\in \Loneloc(\Omega)$ is said to be in $\BMO{\cS_{1}}{p}{(\Omega)}$ if 
$$
\norm{f}{\BMO{\cS_{1}}{p}\!(\Omega)}=\sup_{y\in \tilde{\Omega}}\norm{{f_y}}{\BMO{\cS_{1}}{p}\!(\Omega_1)}<\infty,
$$
where $f_y(x)=f(x,y)$.
\end{definition}

Although this norm combines a supremum with a $\BMO{}{}$ norm, we are justified in calling $\BMO{\cS_{1}}{p}{(\Omega)}$ a $\BMO{}{}$ space as it inherits many properties from $\BMO{\cS}{p}{(\Omega)}$. In particular, for each $i=1,2,\ldots,k$, $\Linfty(\Omega)\subset\BMO{\cS_i}{p}{(\Omega)}$ and $\BMO{\cS_i}{p_1}{(\Omega)}\cong\BMO{\cS_i}{p_2}{(\Omega)}$ for all $1\leq p_1,p_2<\infty$ if the hypothesis of Corollary~\ref{cor-jn} is satisfied with $p_0 = 1$ for 
$\cS_i$ on $\Omega_i$. Moreover, under certain conditions, the spaces $\BMO{\cS_i}{p}{(\Omega)}$ can be quite directly related to $\BMO{\cS}{p}{(\Omega)}$, revealing the product nature of $\BMO{\cS}{p}{(\Omega)}$. This depends on the decomposition property of the basis $\cS$, as well as some differentiation properties of the $\cS_i$.

Before stating the theorem, we briefly recall the main definitions related to the theory of differentiation of the integral; see the survey \cite{guz} for a standard reference. For a basis of shapes $\cS$, denote by $\cS(x)$ the subcollection of shapes that contain $x\in {\Omega}$. We say that $\cS$ is a differentiation basis if for each $x\in \Omega$ there exists a sequence of shapes $\{{S}_k\}\subset\cS(x)$ such that $\delta(S_k)\rightarrow{0}$ as $k\rightarrow\infty$. Here, $\delta(\cdot)$ is the Euclidean diameter. For $f\in \Loneloc(\Omega)$, we define the upper derivative of $\int\!f$ with respect to $\cS$ at $x\in\Omega$ by
$$
\overline{D}(\int\!f,x)=\sup\left\{\limsup_{k\rightarrow\infty}\dashint_{S_k }\!f:\{S_k\}\subset\cS(x),\,\,\delta(S_k)\rightarrow{0}\,\,\text{as}\,\,k\rightarrow\infty  \right\}
$$
and the lower derivative of $\int\!f$ with respect to $\cS$ at $x\in\Omega$ by
$$
\underline{D}(\int\!f,x)=\inf\left\{\liminf_{k\rightarrow\infty}\dashint_{S_k }\!f:\{S_k\}\subset\cS(x),\,\,\delta(S_k)\rightarrow{0}\,\,\text{as}\,\,k\rightarrow\infty  \right\}.
$$
We say, then, that a differentiation basis $\cS$ differentiates $\Loneloc(\Omega)$ if for every $f\in \Loneloc(\Omega)$ and for almost every $x\in \Omega$, $\overline{D}(\int\!f,x)=\underline{D}(\int\!f,x)=f(x)$. The classical Lebesgue differentiation theorem is a statement that the basis $\cB$ (equivalently, $\cQ$) differentiates $\Loneloc(\Omega)$. It is known, however, that the basis $\cR$ does not differentiate $\Loneloc(\Omega)$, but does differentiate the Orlicz space $L(\log L)^{n-1}(\Omega)$ (\cite{jmz}). 

\begin{theorem}
\label{th:4}
Let $\Omega$ be a domain satisfying \eqref{eqn-decomp}, $\cS$ be a basis of shapes for $\Omega$ and $\cS_i$ be a basis of shapes for $\Omega_i$, $1 \leq i \leq k$. 
\begin{enumerate}
\item[a)] Let $f\in \bigcap_{i=1}^{k}{\BMO{\cS_{i}}{p}}(\Omega)$. If $\cS$ satisfies the weak decomposition property with respect to $\{\cS_i\}_{i=1}^{k}$, then $f\in\BMO{{\cS}}{p}{(\Omega)}$ with 
$$
\norm{f}{\BMO{\cS}{p}\!(\Omega)}\leq \sum_{i=1}^{k}\norm{f}{\BMO{\cS_{i}}{p}\!(\Omega)}.
$$
\item[b)] Let $f\in\BMO{{\cS}}{p}{(\Omega)}$. If $\cS$ satisfies the strong decomposition property with respect to $\{\cS_i\}_{i=1}^{k}$ and each $\cS_i$ contains a differentiation basis that differentiates $\Loneloc(\Omega_i)$, then $f\in\bigcap_{i=1}^{k}{\BMO{\cS_{i}}{p}}(\Omega)$ with 
$$
\max_{i=1,\ldots,k}\{\norm{f}{\BMO{\cS_{i}}{p}\!(\Omega)}\}\leq 2^{k-1}\norm{f}{\BMO{\cS}{p}\!(\Omega)}.
$$
When $p = 2$, the constant $2^{k-1}$ can be replaced by $1$.  
\end{enumerate}
\end{theorem}

This theorem was first pointed out in \cite{cs} in the case of $\cR$ in $\T\times\T$. Here, we prove it in the setting of more general shapes and domains, clarifying the role played by the theory of differentiation and keeping track of constants. 

\begin{proof}
We first present the proof in the case of $k=2$. We write $\Omega=X \times Y$, denoting by $(x,y)$ an element in $\Omega$ with $x\in{X}$ and $y\in Y$. The notations $\cS_X$ and $\cS_Y$ will be used for the basis in $X$ and $Y$, respectively. Similarly, the measures $d{x}$ and $d{y}$ will be used for Lebesgue measure on $X$ and $Y$, respectively, while $d{A}$ will be used for the Lebesgue measure on $\Omega$. 

To prove (a), assume that $\cS$ satisfies the weak decomposition property with respect to $\{\cS_X,\cS_Y\}$ and let $f\in\BMO{\cS_X}{p}{(\Omega)}\cap\BMO{\cS_Y}{p}{(\Omega)}$. Fixing a shape $R\in\cS$, write $R=S\times{T}$ for $S\in\cS_X$ and $T\in\cS_Y$. Then, by Minkowski's inequality, 
$$
\left(\dashint_{R}\!|f(x,y)-f_R|^pdA\right)^{\frac{1}{p}}\!\!\leq\left(\dashint_{R}\!|f(x,y)-(f_{y})_S|^pdA\right)^{\frac{1}{p}}\!\!+\left(\dashint_{T}\!|(f_{y})_S-f_R|^pdy\right)^{\frac{1}{p}}.
$$
For the first integral, we estimate
\begin{equation}
\label{eqn-first}
\dashint_{R}\!|f(x,y)-(f_y)_S|^p\,dA\leq \dashint_{T}\!\norm{f_y}{\BMO{\cS_X}{p}\!(X)}^p\,dy\leq \norm{f}{\BMO{\cS_X}{p}\!(\Omega)}^p.
\end{equation}
For the second integral, Jensen's inequality gives
\[
\begin{split}
\left(\dashint_{T}\!|(f_{y})_S-f_R|^p\,dy\right)^{1/p}&\leq \left(\dashint_{T}\! \left| \dashint_{S}\!f_{y}(x)\,dx-\dashint_{S}\dashint_{T}\!f(x,y)\,dy\,dx \right|^pdy\right)^{1/p}\\&= \left(\dashint_{T}\!\left|\dashint_{S}\!\big( f(x,y)-(f_x)_T \big)\,dx \right|^p dy\right)^{1/p}\\&\leq  \left(\dashint_{R}\!|f(x,y)-(f_x)_T|^p\,dA\right)^{1/p}\\&\leq \norm{f}{\BMO{\cS_Y}{p}\!(\Omega)},
\end{split}
\]
where the last inequality follows in the same way as \eqref{eqn-first}. Therefore, we may conclude that $f\in\BMO{{\cS}}{p}{(\Omega)}$ with $\norm{f}{\BMO{\cS}{p}\!(\Omega)}\leq \norm{f}{\BMO{\cS_X}{p}\!(\Omega)}+\norm{f}{\BMO{\cS_Y}{p}\!(\Omega)}$. 

We now come to the proof of (b).  To simplify the  notation, we use $\cO_p(f,S)$ for the $p$-mean oscillation of the function $f$ on the shape $S$, i.e.\
$$\cO_p(f,S): = \dashint_{S}\!|f-f_S|^p.$$

Assume that $\cS$ satisfies the strong decomposition property with respect to $\{\cS_X,\cS_Y\}$, and that $\cS_X$ and $\cS_Y$ each contain a differentiation basis that differentiates $\Loneloc(X)$ and $\Loneloc(Y)$, respectively. 

Let $f\in\BMO{{\cS}}{p}{(\Omega)}$. Fix a shape $S_0\in\cS_X$ and consider the $p$-mean oscillation of $f_y$ on $S_0$, 
$\cO_p(f_y,S_0)$, as a function of $y$. For any $T\in\cS_Y$, writing $R = S_0\times T\in\cS$, we have that $R\in\cS$ by the strong decomposition property of $\cS$, 
so $f\in\BMO{{\cS}}{p}{(\Omega)}$ implies $f \in \Lone(R)$ and therefore
$$
\int_{Y}\!\cO_p(f_y,S_0)\,dy = \frac{1}{|S_0|}\int_{T}\!\int_{S_0}\!|f_y(x)-(f_y)_{S_0}|^p\,dx\,dy\leq \frac{2^p}{|S_0|}\int_{R}\!|f(x,y)|^p\,dx\,dy<\infty.
$$
By Lemma \ref{pr:18}, this is enough to guarantee that $\cO_p(f_y,{S_0})\in \Loneloc(Y)$.

Let  $\eps>0$.  Since $\cS_Y$ contains a differentiation basis,  for almost every $y_0\in Y$ there exists a shape $T_0\in\cS_Y$ containing $y_0$ such that
\begin{equation}
\label{diff}
\left|\dashint_{T_0}\!\cO_p(f_y,{S_0})\,dy-\cO_p(f_{y_0},{S_0}) \right|<\eps.
\end{equation}

Fix such an $x_0$ and a $T_0$ and let $R_0=S_0\times T_0$.  Then the strong decomposition property implies that $R_0\in\cS$,  and by Proposition~\ref{pr:6} applied
to the mean oscillation of $f_y$ on $S_0$, we have
\[
\begin{split}
\dashint_{T_0}\!\cO_p(f_y,{S_0})\,dy
& = \dashint_{T_0}\!\dashint_{{S_0}}|f_y(x)-(f_y)_{S_0}|^p\,dx\,dy\\
& \leq 2^p\dashint_{T_0}\dashint_{{S_0}}\!|f(x,y)-f_{R_0}|^p\,dx\,dy\\
&= 2^p\dashint_{R_0}\!|f(x,y)-f_{R_0}|^p\,dA\\
&\leq 2^p\norm{f}{\BMO{ \cS}{p}(\Omega) }^p.
\end{split}
\]
Note that when $p = 2$, Proposition~\ref{pr:12} implies that the factor of $2^p$ can be dropped.

Combining this with (\ref{diff}), it follows that
$$
\cO_p(f_{y_0},S_0)< \eps+2^p\norm{f}{\BMO{ \cS}{p}(\Omega) }^p.
$$
Taking $\eps\rightarrow{0^+}$, since $S_0$ is arbitrary, this implies that $f_{y_0}\in\BMO{\cS_X}{p}(X)$ with 
\begin{equation}
\label{eqn-partb}
\norm{f_{y_0}}{\BMO{\cS_X }{p}(X) }\leq 2\norm{f}{\BMO{ \cS}{p}(\Omega)}.
\end{equation} 
 The fact that this is true for almost every $y_0\in Y$ implies that 
$\norm{f}{\BMO{\cS_X}{p}(\Omega) }\leq 2\norm{f}{\BMO{\cS}{p}(\Omega)}$.

Similarly, one can show that $\norm{f}{\BMO{\cS_Y}{p}(\Omega) }\leq 2\norm{f}{\BMO{\cS}{p}(\Omega)}$. Thus we have shown $f\in\BMO{{\cS_X}}{p}{(\Omega)}\cap\BMO{{\cS_Y}}{p}{(\Omega)}$ with 
$$\max\{\norm{f}{\BMO{\cS_X}{p}(\Omega) },\norm{f}{\BMO{\cS_Y}{p}(\Omega)}\}\leq 2\norm{f}{\BMO{\cS}{p}(\Omega)}.$$
Again, when $p = 2$ the factor of $2$ disappears.

For $k > 2$, let us assume the result holds for $k - 1$.  Write $X = \Omega_1 \times \Omega_2 \times \ldots \times \Omega_{k-1}$ and $Y = \Omega_k$.
Set $\cS_Y = \cS_k$.  By the weak decomposition property of $\cS$, we can define the projection of the basis $\cS$ onto $X$, namely
\begin{equation}
\label{eqn-cSx}
\cS_X = \{S_1 \times S_2 \times \ldots \times S_{k-1}: S_i \in \cS_i, \exists S_k \in  \cS_k, \prod_{i = 1}^k S_i \in \cS\},
\end{equation}
and this is a basis of shapes on $X$ which by definition also has the weak decomposition property.  Moreover, $\cS$ has the weak decomposition property with respect to $\cS_X$
and $\cS_Y$.  

To prove part (a) for $k$ factors, we first apply the result of part (a) proved above for $k = 2$, followed by the definitions and part (a) applied again to $X$, since we are assuming it is valid with $k - 1$ factors.  This gives us the inclusion 
$\bigcap_{i=1}^{k}{\BMO{\cS_{i}}{p}}(\Omega) \subset \BMO{\cS}{p}(\Omega)$
with the following estimates on the norms (we use the notation $\xihat$ for the $k-2$ tuple of variables obtained from $(x_1,\ldots, x_{k-1})$ by removing $x_i$): 
\[
\begin{split}
\norm{f}{\BMO{\cS}{p}(\Omega)} 
& \leq \norm{f}{\BMO{\cS_X}{p}(\Omega)}+\norm{f}{\BMO{\cS_Y}{p}(\Omega)}\\
& = \sup_{y\in Y}\norm{{f_y}}{\BMO{\cS_{X}}{p}\!(X)}+\norm{f}{\BMO{\cS_Y}{p}(\Omega)}\\
& \leq \sup_{y\in Y}\sum_{i=1}^{k-1}\norm{f_y}{\BMO{\cS_{i}}{p}\!(X)}+\norm{f}{\BMO{\cS_Y}{p}(\Omega)}\\
&= \sup_{y\in Y}\sum_{i=1}^{k-1}\sup_{\xihat}\norm{({f_y})_\xihat}{\BMO{\cS_{i}}{p}\!(\Omega_i)}+\norm{f}{\BMO{\cS_Y}{p}(\Omega)}\\
& \leq \sum_{i=1}^{k-1}\sup_{(\xihat,y)}\norm{f_{(\xihat,y)}}{\BMO{\cS_{i}}{p}\!(\Omega_i)}+\norm{f}{\BMO{\cS_Y}{p}(\Omega)}\\
& = \sum_{i=1}^{k-1}\norm{f}{\BMO{\cS_{i}}{p}\!(\Omega)}+\norm{f}{\BMO{\cS_k}{p}(\Omega)}\\
& = \sum_{i=1}^{k}\norm{f}{\BMO{\cS_{i}}{p}\!(\Omega)}.
\end{split}
\]

To prove part (b) for $k > 2$, we have to be more careful. First note that if $\cS$ has the strong decomposition property, then so does $\cS_X$ defined by \eqref{eqn-cSx}.  We repeat the first part of the proof of (b) for the case $k = 2$ above, with $X = \Omega_1 \times \Omega_2 \times \ldots \times \Omega_{k-1}$ and $Y = \Omega_k$, leading up to the estimate \eqref{eqn-partb} for the function $f_{y_0}$ for some $y_0 \in Y$. Note that in this part we only used the differentiation properties of $Y$, which hold by hypothesis in this case since $Y = \Omega_k$.  Now we repeat the process for the function $f_{y_0}$ instead of $f$, with $X_1 = \Omega_1 \times \Omega_2 \times \ldots \times \Omega_{k-2}$ and $Y_1 = \Omega_{k-1}$. This gives
$$\norm{(f_{y_0})_{y_1}}{\BMO{\cS_{X_1} }{p}(X_1) }\leq 2\norm{f_{y_0}}{\BMO{ \cS}{p}(X)} \leq 4 \norm{f}{\BMO{ \cS}{p}(\Omega)} \quad \forall y_1 \in \Omega_{k-1}, y_0 \in \Omega_k.$$ 
We continue until we get to $X_{k-1} = \Omega_1$, for which $\cS_{X_k} = \cS_1$, yielding the estimate
$$\norm{f_{({y_{k-2}, \ldots, y_0)}}}{\BMO{\cS_1 }{p}(\Omega_1)}  \leq \ldots \leq 2^{k-2}\norm{f_{y_0}}{\BMO{ \cS}{p}(X)}
\leq 2^{k-1} \norm{f}{\BMO{ \cS}{p}(\Omega)}$$
for all $k-1$-tuples $y = (y_{k-2}, \ldots, y_0) \in \Omegatil = \Omega_2 \times \ldots \Omega_k$.
Taking the supremum over all such $y$, we have, by Definition~\ref{def-cS1}, that $f \in \BMO{\cS_{1}}{p}\!(\Omega)$ with
$$
\norm{f}{\BMO{\cS_{1}}{p}\!(\Omega)}=\sup_{y\in \Omegatil}\norm{{f_y}}{\BMO{\cS_{1}}{p}\!(\Omega_1)} \leq 2^{k-1} \norm{f}{\BMO{ \cS}{p}(\Omega)}.
$$
A similar process for $i = 2, \ldots k$ shows that $f \in \BMO{\cS_{i}}{p}\!(\Omega)$ with
$
\norm{f}{\BMO{\cS_{i}}{p}\!(\Omega)} \leq 2^{k-1} \norm{f}{\BMO{ \cS}{p}(\Omega)}.
$
As the factor of $2$ appears in the proof for $k = 2$ only when $p \neq 2$, the same will happen here.
\end{proof}

Since $\cQ$ satisfies the weak decomposition property, the claim of part (a) holds for $\BMO{}{}$, a fact pointed out in \cite{st} without proof. Also, it is notable that there was no differentiation assumption required for this direction.

In the proof of part (b), differentiation is key and the strong decomposition property of the basis cannot be eliminated as there would be no guarantee that arbitrary $S$ and $T$ would yield a shape $R$ in $\Omega$. In fact, if the claim were true for bases with merely the weak decomposition property, this would imply that $\BMO{}{}$ and $\BMO{\cR}{}$ are congruent, which is not true (see Example \ref{ex:1}). 

\section*{{\bf Acknowledgements}}

The authors would like to thank Almut Burchard for discussions that initiated much of the investigation that lead to this paper. 

\bibliographystyle{amsplain}

\end{document}